\documentclass[11pt,reqno]{amsart}

\usepackage[T1]{fontenc}
\usepackage{lmodern}
\usepackage{microtype}
\usepackage[a4paper,left=0.86in,right=0.86in,top=0.82in,bottom=0.88in]{geometry}
\usepackage{amsmath,amssymb,mathtools}
\usepackage{hyperref}
\usepackage[nameinlink,noabbrev]{cleveref}

\hypersetup{
  colorlinks=true,
  linkcolor=blue,
  citecolor=blue,
  urlcolor=blue,
  pdftitle={Sharp mean-width and Jacobian bounds for Euclidean and hyperbolic harmonic maps},
  pdfauthor={Deguang Zhong and David Kalaj}
}

\newtheorem{theorem}{Theorem}[section]
\newtheorem{corollary}[theorem]{Corollary}
\newtheorem{proposition}[theorem]{Proposition}
\newtheorem{lemma}[theorem]{Lemma}
\theoremstyle{remark}
\newtheorem{remark}[theorem]{Remark}

\newcommand{\R}{\mathbb R}
\newcommand{\B}{\mathbb B}
\newcommand{\Sph}{\mathbb S}

\newcommand{\co}{\operatorname{co}}

\newcommand{\Ph}{P_h}
\newcommand{\Dh}{\Delta_h}

\newcommand{\id}{\mathrm{id}}
\newcommand{\tr}{\operatorname{tr}}

\title[Mean-width and Jacobian bounds for harmonic maps]
{Sharp mean-width and Jacobian bounds\\
for Euclidean and hyperbolic harmonic maps}

\author{Deguang Zhong}
\address{Institute of Applied Mathematics, Shenzhen Polytechnic University, Shenzhen 518055, China}
\email{huachengzhon@163.com}
\author{David Kalaj}
\address{Faculty of Natural Sciences and Mathematics, University of Montenegro, 81000 Podgorica, Montenegro}
\email{davidk@ucg.ac.me}

\date{September 6, 2026}

\begin{document}

\begin{abstract}
We prove a sharp mean-width inequality for monotone zonal operators and
derive global and differential bounds for Euclidean and hyperbolic-harmonic
self-maps of the unit ball.  Let $k:[0,\pi]\to\mathbb R$ be continuous
and nonincreasing, and let $T_k$ be the associated zonal integral operator,
\[
 (T_kF)(\xi)=\int_{\mathbb S^{n-1}}
 k\!\left(\arccos\langle\xi,\eta\rangle\right)
 F(\eta)\,d\sigma(\eta),
\]
where $\sigma$ is normalized surface measure.  For every measurable
$F:\mathbb S^{n-1}\to\overline{\mathbb B^n}$, we prove
\[
 w\!\left(\operatorname{co}T_kF(\mathbb S^{n-1})\right)
 \le2\lambda_1(k),
\]
where $w$ denotes mean width, normalized by $w(\overline{\mathbb B^n})=2$,
and $\lambda_1(k)$ is the eigenvalue of $T_k$ on the space of spherical
harmonics of degree one.  For strictly decreasing kernels, equality
holds exactly for $F(\eta)=Q\eta$ almost everywhere, with $Q\in O(n)$.
For the ordinary Poisson kernel, the multiplier is $r$, yielding sharp
mean-width, intrinsic-volume, and image-volume contraction.  In particular,
$|f(r\mathbb B^n)|\le\omega_n r^n$ without injectivity assumptions,
answering the area and higher-dimensional volume question of Koh and Kovalev.

For componentwise hyperbolic-harmonic self-maps, the corresponding sharp
comparison radius is
\[
 \Lambda_n(r)=\frac{2(n-1)}n\,r\,
 {}_2F_1\!\left(1,1-\frac n2;\frac{n+2}2;r^2\right).
\]
Thus the image volume is at most $\omega_n\Lambda_n(r)^n$, with equality
precisely for Poisson--Szeg\H{o} extensions of orthogonal boundary data.
Hyperbolic domain automorphisms also give sharp bounds for the mean width
of the derivative ellipsoid and for the determinant at every point.  For ordinary harmonic maps, we obtain an exact
variational formula for the pointwise Jacobian supremum, reduce it to a
one-variable minimization with a unique minimizer, and prove strict
anisotropy of the extremal
derivative away from the origin when $n\ge3$.  An explicit
Poisson-information bound is strictly larger than the exact supremum
in that regime.
Applications include zonal semigroups, Henyey--Greenstein scattering,
and bounded zonoid depth regions.
\end{abstract}

\maketitle

\section{Introduction}

\subsection{Historical background and the area-contraction problem}

Harmonic mappings connect potential theory with geometric function theory.
For planar univalent mappings, the modern theory begins with Clunie and
Sheil-Small \cite{ClunieSheilSmall}; Lewy's theorem \cite{Lewy} gives the
classical local Jacobian criterion, and Duren \cite{Duren} provides a
systematic account.  Harmonic analogues of the Schwarz lemma give
pointwise and differential bounds: see Heinz \cite{Heinz}, Hethcote
\cite{Hethcote}, Colonna \cite{Colonna}, Chen \cite{Chen},
Kalaj--Vuorinen \cite{KalajVuorinen}, and, in higher dimensions,
Dai--Pan \cite{DaiPan2015,DaiPan}.  Our global estimates concern the
size of image sets.

The area-contraction problem asks how a harmonic mapping changes the
area of a concentric subdisk.  Koh and Kovalev proved that if
\[
 f:\mathbb D\to\mathbb D
\]
is a harmonic self-homeomorphism, then
\[
 |f(r\mathbb D)|\le \pi r^2,\qquad 0<r<1;
\]
see \cite{KohKovalev}.   Question~1 from \cite{KohKovalev}  asks whether the same comparison
holds for arbitrary harmonic self-maps and whether an analogous
statement holds in higher dimensions.
The planar question is also attributed to Koh and Kovalev in
Problem~3.25 of the list compiled by Bshouty and Lyzzaik
\cite{BshoutyLyzzaik}.  A recent preprint of Geleta \cite{Geleta}
claims the planar area inequality for sense-preserving univalent
self-maps of the disk.  Our argument is independent and uses a mean-width
estimate for convex hulls.  It applies without injectivity or orientation
assumptions in every dimension $n\ge2$.

The corresponding question for Euclidean balls is the following.  If
\[
 f:\mathbb B^n\longrightarrow\mathbb B^n,
 \qquad n\ge2,
\]
is harmonic coordinatewise, must one have
\begin{equation}\label{eq:intro-volume-question}
 |f(r\mathbb B^n)|
 \le
 |r\mathbb B^n|
 =
 \omega_n r^n,
 \qquad 0<r<1,
\end{equation}
without any injectivity assumption?  A Jacobian argument must account for multiplicity: the area formula
counts repeated coverings, whereas the Lebesgue measure of
$f(r\mathbb B^n)$ counts the image only once.  For example, write a planar harmonic map as $f=h+\overline g$, with
$h,g$ holomorphic, so that $J_f=|h'|^2-|g'|^2$.
Already for the analytic map $f(z)=z^N$, with $N\ge2$,
\[
 \int_{\mathbb D_r}J_f\,dA
 =
 N\pi r^{2N},
 \qquad
 |f(\mathbb D_r)|
 =
 \pi r^{2N}.
\]
Thus an identity between image area and the integral of the Jacobian
cannot be used without accounting for multiplicity.

We obtain the volume bound from a stronger statement: the convex hull of
the spherical trace has mean width at most that of the radius-$r$ ball.
The harmonic convex-hull property and Urysohn's inequality then control
the image volume.  The underlying operator estimate uses radial monotonicity and therefore
applies to a broader class of zonal kernels.
We therefore state the zonal theorem first and derive the harmonic
results as consequences.

\subsection{A sharp contraction principle for monotone zonal operators}

For a nonempty compact convex set $K\subset\mathbb R^n$, let
\[
 h_K(u)=\sup_{x\in K}\langle x,u\rangle,
 \qquad u\in\mathbb S^{n-1},
\]
and let $\sigma$ denote normalized spherical measure.  We use the
normalization
\begin{equation}\label{eq:introMeanWidth}
 w(K)
 =
 \int_{\mathbb S^{n-1}}
 \bigl(h_K(u)+h_K(-u)\bigr)\,d\sigma(u)
 =
 2\int_{\mathbb S^{n-1}}h_K(u)\,d\sigma(u),
\end{equation}
so that $w(R\overline{\mathbb B^n})=2R$.

Let
\[
 d(\xi,\eta)=\arccos\langle\xi,\eta\rangle.
\]
Given a continuous nonincreasing function
$k:[0,\pi]\to\mathbb R$, define the zonal operator
\[
 T_kF(\xi)
 =
 \int_{\mathbb S^{n-1}}
 k(d(\xi,\eta))F(\eta)\,d\sigma(\eta).
\]
The same notation will be used for the corresponding operator on
scalar-valued functions.

Let
\[
 \mathcal H_1(\mathbb S^{n-1})
 =
 \left\{
 Y_a:\mathbb S^{n-1}\to\mathbb R:
 Y_a(\eta)=\langle a,\eta\rangle,\quad a\in\mathbb R^n
 \right\}.
\]
Thus $\mathcal H_1(\mathbb S^{n-1})$ is the space of spherical
harmonics of degree one and
\[
 \dim \mathcal H_1(\mathbb S^{n-1})=n.
\]

\begin{lemma}[The degree-one multiplier]
\label{lem:degree-one-multiplier}
Let $n\ge2$ and let $k:[0,\pi]\to\mathbb R$ be continuous.  There exists
a unique scalar $\lambda_1(k)$ such that
\[
 T_kY=\lambda_1(k)Y
 \qquad
 \text{for every }Y\in\mathcal H_1(\mathbb S^{n-1}).
\]
Equivalently,
\[
 T_k|_{\mathcal H_1}
 =
 \lambda_1(k)I_{\mathcal H_1}.
\]
Moreover,
\begin{equation}\label{eq:intro-lambda1}
 \lambda_1(k)
 =
 c_n\int_0^\pi
 k(t)\cos t\,\sin^{n-2}t\,dt,
 \qquad
 c_n=\frac{|\mathbb S^{n-2}|}{|\mathbb S^{n-1}|}.
\end{equation}
\end{lemma}

\begin{proof}
Set
\[
 V_k(\xi)
 =
 \int_{\mathbb S^{n-1}}
 k(d(\xi,\eta))\,\eta\,d\sigma(\eta).
\]
Rotational invariance of the kernel and of $\sigma$ gives
\[
 V_k(Q\xi)=QV_k(\xi),
 \qquad Q\in O(n).
\]
For fixed $\xi$, every orthogonal transformation fixing $\xi$ also fixes
$V_k(\xi)$.  Hence $V_k(\xi)$ is a scalar multiple of $\xi$.
Rotational covariance shows that this scalar is independent of $\xi$;
thus
\[
 V_k(\xi)=\lambda_1(k)\xi.
\]
Taking $\xi=e_n$ and the scalar product with $e_n$ gives
\[
 \lambda_1(k)
 =
 \int_{\mathbb S^{n-1}}
 k(d(e_n,\eta))\eta_n\,d\sigma(\eta),
\]
which in spherical coordinates is exactly
\eqref{eq:intro-lambda1}.

Finally, for
\[
 Y_a(\eta)=\langle a,\eta\rangle\in\mathcal H_1(\mathbb S^{n-1}),
\]
we have
\[
 \begin{aligned}
 T_kY_a(\xi)
 &=
 \int_{\mathbb S^{n-1}}
 k(d(\xi,\eta))\langle a,\eta\rangle\,d\sigma(\eta)\\
 &=
 \langle a,V_k(\xi)\rangle\\
 &=
 \lambda_1(k)\langle a,\xi\rangle
 =
 \lambda_1(k)Y_a(\xi).
 \end{aligned}
\]
\end{proof}

See, for example, \cite{DaiXu} for background on spherical harmonics
and zonal operators.  In particular, for $n=2$ one has $c_2=1/\pi$ and
$\sin^{n-2}t\equiv1$, whereas for $n\ge3$ the factor
$\sin^{n-2}t$ is the spherical-coordinate weight.

The basic operator estimate is as follows.

\begin{theorem}[Sharp monotone zonal contraction]\label{thm:zonal-main}
Let $n\ge2$, let $k:[0,\pi]\to\mathbb R$ be continuous and nonincreasing,
and let
\[
 F:\mathbb S^{n-1}\to\overline{\mathbb B^n}
\]
be measurable.  Then
\begin{equation}\label{eq:zonal-main}
 w\!\left(
 \operatorname{co}T_kF(\mathbb S^{n-1})
 \right)
 \le2\lambda_1(k).
\end{equation}
The constant is sharp.
\end{theorem}

The sharp bound therefore depends only on the degree-one multiplier.  Radial monotonicity is
essential: later we construct a continuous nonnegative rotationally
invariant Markov kernel with $\lambda_1=0$ for which the left-hand side of
\eqref{eq:zonal-main} is strictly positive.

The equality case is rigid.

\begin{theorem}[Rigidity for strictly decreasing zonal kernels]
\label{thm:zonal-rigidity}
Under the hypotheses of Theorem~\ref{thm:zonal-main}, assume in addition
that $k$ is strictly decreasing.  Then
\[
 w\!\left(
 \operatorname{co}T_kF(\mathbb S^{n-1})
 \right)
 =
 2\lambda_1(k)
\]
if and only if
\[
 F(\eta)=Q\eta
 \qquad\text{for a.e. }\eta\in\mathbb S^{n-1}
\]
for some $Q\in O(n)$.
\end{theorem}

\subsection{Harmonic consequences}

For $0<r<1$, the ordinary Poisson kernel of $\mathbb B^n$ has the
zonal form
\[
 k(t)=p_r(t)
 =
 \frac{1-r^2}{(1-2r\cos t+r^2)^{n/2}},
 \qquad 0\le t\le\pi.
\]
By \eqref{eq:intro-lambda1}, its degree-one multiplier is
\begin{equation}\label{eq:intro-euclidean-first-multiplier}
 \lambda_1(p_r)
 =
 c_n\int_0^\pi
 \frac{(1-r^2)\cos t\,\sin^{n-2}t}
 {(1-2r\cos t+r^2)^{n/2}}\,dt
 =r.
\end{equation}
The last identity is equivalently the Poisson reproduction of degree-one
harmonic functions; see \cite{AxlerBourdonRamey} for standard Poisson theory
and \cite{DaiXu} for the spherical-harmonic viewpoint.  Thus the first
multiplier is $r$ in every dimension.  Since $p_r$ is strictly decreasing,
Theorem~\ref{thm:zonal-main} gives the following contraction theorem with the
same sharp constant in every dimension.

\begin{theorem}[Sharp convex-hull mean-width contraction]
\label{thm:main}
Let $n\ge2$ and let
$f:\mathbb B^n\to\mathbb B^n$ be harmonic.  Then, for every $0<r<1$,
\begin{equation}\label{eq:main}
 w\!\left(\operatorname{co}f(r\mathbb S^{n-1})\right)\le2r.
\end{equation}
The constant is sharp.
\end{theorem}

In dimension two, Cauchy's perimeter formula gives
$\operatorname{Per}(K)=\pi w(K)$.  For degenerate convex sets, perimeter
is understood by continuous extension; in particular, a segment has
perimeter twice its length.  Hence:

\begin{corollary}[Planar perimeter and area contraction]
\label{cor:planar}
Let $f:\mathbb D\to\mathbb D$ be harmonic.  Then
\begin{equation}\label{eq:planar-perimeter}
 \operatorname{Per}\!\left(\operatorname{co}f(r\mathbb T)\right)
 \le2\pi r,
 \qquad 0<r<1,
\end{equation}
and
\begin{equation}\label{eq:planar-area}
 |f(r\mathbb D)|\le\pi r^2.
\end{equation}
Both constants are sharp.
\end{corollary}

More generally, if
\[
 K_r=\operatorname{co}f(r\mathbb S^{n-1}),
\]
then all intrinsic volumes contract.

\begin{corollary}[Intrinsic-volume contraction]
\label{cor:intrinsic}
For every $n\ge2$, every harmonic
$f:\mathbb B^n\to\mathbb B^n$, and every $j=1,\ldots,n$,
\begin{equation}\label{eq:intro-intrinsic}
 V_j(K_r)
 \le
 V_j(r\overline{\mathbb B^n})
 =
 r^jV_j(\overline{\mathbb B^n}).
\end{equation}
Each inequality is sharp.
\end{corollary}

Taking $j=n$ and using the harmonic convex-hull property gives

\begin{corollary}[Sharp volume contraction]
\label{cor:volume}
For every $n\ge2$ and every harmonic
$f:\mathbb B^n\to\mathbb B^n$,
\begin{equation}\label{eq:volume}
 |f(r\mathbb B^n)|
 \le
 \omega_n r^n
 =
 |r\mathbb B^n|,
 \qquad 0<r<1.
\end{equation}
The constant is sharp.
\end{corollary}

Corollaries~\ref{cor:planar} and \ref{cor:volume} answer both parts of
Koh--Kovalev's Question~1.  The planar conclusion also resolves
Problem~3.25 in \cite{BshoutyLyzzaik}.  The intrinsic-volume inequalities
show that the volume comparison follows from a single mean-width bound.

Equality at a single radius determines the harmonic map throughout the ball.

\begin{theorem}[One-radius rigidity for harmonic self-maps]
\label{thm:harmonic-rigidity}
Let $f:\mathbb B^n\to\mathbb B^n$ be harmonic.  If for some
$r_0\in(0,1)$,
\[
 w\!\left(
 \operatorname{co}f(r_0\mathbb S^{n-1})
 \right)
 =
 2r_0,
\]
then there exists $Q\in O(n)$ such that
\[
 f(x)=Qx,\qquad x\in\mathbb B^n.
\]
Conversely, every orthogonal map has equality for every radius.
\end{theorem}

\subsection{Hyperbolic-harmonic consequences}\label{sec:intro-hyperbolic}

The same zonal principle applies to the Poisson--Szeg\H{o} kernel of the real hyperbolic ball.  For $n\ge2$, set
\begin{equation}\label{eq:hyp-laplacian-intro}
 \Delta_h=(1-|x|^2)^2\Delta+2(n-2)(1-|x|^2)x\cdot\nabla,
\end{equation}
and call a $C^2$ map $u=(u_1,\ldots,u_n):\mathbb B^n\to\mathbb R^n$
\emph{hyperbolic harmonic} if $\Delta_hu_j=0$ for every component.
This is a linear, componentwise equation on the hyperbolic source; all
target norms, convex hulls, and volumes in this paper are Euclidean.
It is distinct from the nonlinear harmonic-map equation between two
hyperbolic manifolds.  When $n=2$, it is equivalent to ordinary harmonicity;
the new distortion profile arises for $n\ge3$.

For $0\le r<1$, write
\[k(t)=
 p_r^h(t)
 =
 \left(\frac{1-r^2}{1-2r\cos t+r^2}\right)^{n-1},
 \qquad 0\le t\le\pi,
\]
for the zonal form of the Poisson--Szeg\H{o} kernel.  Its degree-one
spherical multiplier is
\begin{equation}\label{eq:hyp-Lambda-intro}
\begin{aligned}
 \Lambda_n(r)
 :=\lambda_1(p_r^h)
 &=
 c_n\int_0^\pi
 \left(\frac{1-r^2}{1-2r\cos t+r^2}\right)^{n-1}
 \cos t\,\sin^{n-2}t\,dt\\
 &=
 \int_{\mathbb S^{n-1}}P_h(re_n,\zeta)\,\zeta_n\,d\sigma(\zeta),
\end{aligned}
\end{equation}
where
\begin{equation}\label{eq:hyp-Ph}
 P_h(x,\zeta)=\left(\frac{1-|x|^2}{|x-\zeta|^2}\right)^{n-1}.
\end{equation}
Thus $\Lambda_n(r)$ is precisely the degree-one multiplier
\eqref{eq:intro-lambda1} of the Poisson--Szeg\H{o} kernel.  We shall prove
\begin{equation}\label{eq:hyp-Lambda-explicit-intro}
 \Lambda_n(r)=\frac{2(n-1)}{n}\,r\,
 {}_2F_1\!\left(1,1-\frac n2;\frac{n+2}{2};r^2\right).
\end{equation}
The hyperbolic counterpart of Theorem~\ref{thm:main} is the following.

\begin{theorem}[Sharp hyperbolic-harmonic mean-width and volume distortion]\label{thm:hyp-main-intro}
Let $n\ge3$ and let $u:\mathbb B^n\to\mathbb B^n$ be hyperbolic harmonic. Then, for every $0<r<1$,
\[
 w\!\left(\operatorname{co}u(r\mathbb S^{n-1})\right)\le2\Lambda_n(r)
\]
and
\[
 |u(r\mathbb B^n)|\le\omega_n\Lambda_n(r)^n.
\]
Both estimates are sharp.  Equality in either estimate at one radius
occurs precisely for the Poisson--Szeg\H{o} extensions of orthogonal
boundary data.
\end{theorem}

The extremal associated with $Q\in O(n)$ is
\[
 U_Q(x)=P_h[Q|_{S^{n-1}}](x)
 =\Lambda_n(|x|)Q\frac{x}{|x|},\qquad x\ne0,
\]
with $U_Q(0)=0$.  Since $\Lambda_n'(0)=2(n-1)/n>1$ for $n\ge3$, the Euclidean volume contraction $|u(r\mathbb B^n)|\le\omega_nr^n$ fails in the hyperbolic-harmonic setting.

Precomposition with hyperbolic automorphisms transfers the derivative
bound at the origin to any prescribed point.

\begin{theorem}[Sharp invariant hyperbolic differential form]\label{thm:hyp-differential-intro}
Let $n\ge2$ and let $u:\mathbb B^n\to\mathbb B^n$ be hyperbolic harmonic.
Then, for every $x\in\mathbb B^n$,
\[
 \int_{\mathbb S^{n-1}}|Du(x)^T\xi|\,d\sigma(\xi)
 \le \frac{2(n-1)}{n(1-|x|^2)},
\]
and
\[
 |\det Du(x)|\le
 \left(\frac{2(n-1)}{n(1-|x|^2)}\right)^n.
\]
Both constants are sharp at every prescribed point.
\end{theorem}

\subsection{The pointwise Euclidean Jacobian problem}

For ordinary harmonic maps in dimensions $n\ge3$, precomposition with a
ball automorphism does not preserve harmonicity.  We therefore use the
Poisson gradient to determine the pointwise determinant supremum directly.
For fixed $x\in\mathbb B^n$, put
\[
 g_x(\eta)=\nabla_xP(x,\eta),\qquad
 \mathcal L_x(C)=\int_{\mathbb S^{n-1}}|Cg_x(\eta)|\,d\sigma(\eta),
\]
where $P$ is the ordinary Poisson kernel and $C$ is symmetric positive definite.
Theorem~\ref{local:thm:exact-euclidean-jacobian} proves the exact formula
\[
 \sup_{f:\mathbb B^n\to\mathbb B^n\ \mathrm{harmonic}}|\det Df(x)|
 =\frac1{n^n}\min_{\substack{C>0\\\det C=1}}\mathcal L_x(C)^n.
\]
The minimizing matrix produces explicit extremal boundary data.
Rotational symmetry reduces the matrix problem to a one-variable minimum
with a unique minimizer $q_n(r)$, where $r=|x|$.
For $n\ge3$ and $r>0$, we prove $q_n(r)<1$: the extremal derivative has
a larger radial singular value than its tangential singular values.
In dimension two, the profile is $(1-r^2)^{-2}$.

The same variational formula gives a closed upper bound in terms of the
Poisson-information matrix.  This bound is sharp at the origin and
throughout the disk, but is strictly larger than the exact supremum when
$n\ge3$ and $x\ne0$.  The explicit bound is convenient for direct estimates, while the
variational formula identifies the sharp constant and its extremals.

\subsection{The strategy of the proof}

The global estimates rest on a sharp inequality for spherical trimmed bodies.
For a measurable map
\[
 F:\mathbb S^{n-1}\to\overline{\mathbb B^n}
\]
and $0\le s\le1$, consider
\[
 \mathcal Z_s(F)
 =
 \left\{
 \int_{\mathbb S^{n-1}}a(\eta)F(\eta)\,d\sigma(\eta):
 0\le a\le1,\quad \int_{\mathbb S^{n-1}} a(\eta)\,d\sigma(\eta)=s
 \right\}.
\]
If $\mu=F_{\#}\sigma$ denotes the push-forward of $\sigma$ under $F$, then
$\mathcal Z_s(F)$ is the unnormalized $s$-section of the lift zonoid of
$\mu$, equivalently $\mathcal Z_s(F)=sD_s(\mu)$ for $s>0$;
see \cite[Sec.~2, Def.~1.1., Def.~2.1 and Prop.~2.2, pp.~1999--2002]{KoshevoyMosler1997}.

If
\[
 A_n(\alpha)
 =
 c_n\int_0^\alpha\sin^{n-2}t\,dt,
 \qquad
 \alpha_n=A_n^{-1},
\]
define
\[
 \Phi_n(s)
 =
 \frac{c_n}{n-1}\sin^{n-1}\alpha_n(s).
\]
We prove
\[
 w(\mathcal Z_s(F))
 \le
 2\Phi_n(s).
\]
The directional extremizers are spherical caps.  The proof uses the
bathtub principle in the form recorded in \cite{LiebLoss}, followed by
Fubini and the strict concavity encoded in
\[
 \Phi_n'(s)=\cos\alpha_n(s).
\]

A continuous nonincreasing zonal kernel admits a radial layer-cake
decomposition into spherical caps.  Combining that decomposition with the
trimmed-body inequality leaves precisely the degree-one multiplier
$\lambda_1(k)$.  For the Poisson kernel this multiplier is $\rho$, which
produces Theorem~\ref{thm:main}.  Schematically,
\begin{align*}
 \text{spherical-cap extremality}
 &\longrightarrow
 \text{trimmed-body inequality}\\
 &\longrightarrow
 \text{radial layer cake}\\
 &\longrightarrow
 \text{degree-one spherical harmonic}\\
 &\longrightarrow
 \text{sharp mean-width contraction}.
\end{align*}

The rigidity theorem is obtained by tracing equality through every step:
strict Jensen equality forces the trimming density to be constant, the
bathtub optimizer is unique, and equality in the layer-cake estimate
forces all one-dimensional projections of the boundary datum to be first
spherical harmonics.

Two exact applications illustrate that the mechanism is not specific to
harmonic maps.  In dimension three the Henyey--Greenstein scattering
kernel is precisely the Poisson kernel after normalization of solid angle,
while the trimmed-body inequality gives sharp envelopes for zonoid depth
regions of bounded multivariate distributions.

The paper is organized as follows.  Section~2 collects the convex-geometric and spherical preliminaries.  Section~3 proves the sharp trimmed-body inequality, and Section~4 establishes the monotone-zonal contraction theorem, rigidity, the semigroup consequence, and the necessity of radial monotonicity.  Sections~5--7 treat the Euclidean Poisson kernel, harmonic self-maps, and their convex-geometric consequences.  Section~8 develops the hyperbolic Poisson--Szeg\H{o} specialization, including the explicit first-mode profile, rigidity, pseudohyperbolic balls, and sharp differential estimates.  Section~9 solves the exact pointwise Jacobian extremal problem for ordinary Euclidean harmonic self-maps.  Section~10 records the scattering and zonoid-depth applications.

\section{Preliminaries}

We use standard facts from convex geometry, harmonic function theory, and
spherical harmonic analysis.  Our conventions are compatible with
\cite{Schneider} for convex bodies, \cite{AxlerBourdonRamey} for harmonic
functions in the ball, and \cite{DaiXu} for spherical harmonics and zonal
operators.

Throughout the rest of the paper, $n\ge2$.  We use the convention
$|\mathbb S^0|=2$.  We write $\mathbb B^n=\{x\in\mathbb R^n:|x|<1\}$,
$\mathbb S^{n-1}=\partial\mathbb B^n$, and $\omega_n=|\mathbb B^n|$.
Harmonicity of a vector-valued map means harmonicity of each coordinate.
The notation $\operatorname{co}$ denotes convex hull; all hulls of
continuous spherical images considered below are compact.
The symbol $\sigma$ denotes the normalized rotation-invariant
probability measure on $\mathbb S^{n-1}$:
\[
 \sigma(\mathbb S^{n-1})=1.
\]
If $dS$ denotes the usual surface measure, then
\[
 d\sigma=\frac{dS}{|\mathbb S^{n-1}|}.
\]

\subsection{Support functions and mean width}

For a nonempty compact convex set $K\subset\mathbb R^n$, its support function
is
\[
 h_K(u)=\sup_{x\in K}\langle x,u\rangle,\qquad u\in\mathbb S^{n-1}.
\]
We define
\[
 w(K)=\int_{\mathbb S^{n-1}}
       \bigl(h_K(u)+h_K(-u)\bigr)\,d\sigma(u).
\]
By invariance of $\sigma$ under $u\mapsto-u$,
\begin{equation}\label{eq:meanwidth}
 w(K)=2\int_{\mathbb S^{n-1}}h_K(u)\,d\sigma(u).
\end{equation}
This normalization gives $w(R\overline{\mathbb B^n})=2R$.

We shall use Urysohn's inequality in the form
\begin{equation}\label{eq:urysohn}
 \left(\frac{|K|}{\omega_n}\right)^{1/n}\le \frac{w(K)}2.
\end{equation}
For convex bodies this is classical; see, for example,
\cite{Schneider}.  If $K$ has empty interior, then $|K|=0$ and the volume
estimate needed below is automatic.  Alternatively, one may apply
\eqref{eq:urysohn} to $K+\varepsilon\overline{\mathbb B^n}$ and let
$\varepsilon\downarrow0$.

\subsection{Intrinsic volumes}

We recall standard facts from convex geometry and use the normalization of
Schneider \cite[Chapter~4]{Schneider}.  For $0\le j\le n$, let $V_j(K)$
denote the $j$th intrinsic volume of a nonempty compact convex set
$K\subset\mathbb R^n$.  In the standard Steiner normalization
\cite[Theorem~4.2.1]{Schneider},
\begin{equation}\label{eq:Steiner}
 |K+t\overline{\mathbb B^n}|
 =
 \sum_{j=0}^n
 \omega_{n-j}V_j(K)t^{\,n-j},
 \qquad t\ge0,
\end{equation}
where $\omega_k=|\mathbb B^k|$ and $\omega_0=1$.
Thus
\[
 V_n(K)=|K|,
 \qquad
 V_0(K)=1,
\]
and, when $K$ has nonempty interior,
\[
 2V_{n-1}(K)=\mathcal H^{n-1}(\partial K).
\]
For the Euclidean unit ball,
\begin{equation}\label{eq:intrinsic-ball}
 V_j(\overline{\mathbb B^n})
 =
 \binom{n}{j}\frac{\omega_n}{\omega_{n-j}},
 \qquad 0\le j\le n.
\end{equation}

The first intrinsic volume is proportional to mean width.  With the
normalization \eqref{eq:meanwidth},
\begin{equation}\label{eq:V1-meanwidth}
 \frac{V_1(K)}{V_1(\overline{\mathbb B^n})}
 =
 \frac{w(K)}2.
\end{equation}

The Alexandrov--Fenchel inequalities imply the monotonicity of the
normalized intrinsic volumes; see, for example,
\cite[Sections~7.3--7.4]{Schneider}.  Together with
\eqref{eq:V1-meanwidth}, this yields
\begin{equation}\label{eq:extended-Urysohn}
 \left(
 \frac{V_j(K)}{V_j(\overline{\mathbb B^n})}
 \right)^{1/j}
 \le
 \frac{V_1(K)}{V_1(\overline{\mathbb B^n})}
 =
 \frac{w(K)}2,
 \qquad 1\le j\le n.
\end{equation}
For lower-dimensional compact convex sets the same inequalities follow by
Hausdorff approximation with $K+\varepsilon\overline{\mathbb B^n}$.

\subsection{Spherical caps}

Put
\begin{equation}\label{eq:cn}
 c_n:=\frac{|\mathbb S^{n-2}|}{|\mathbb S^{n-1}|}.
\end{equation}
For $v\in\mathbb S^{n-1}$ and $0\le\alpha\le\pi$, let
\[
 C(v,\alpha)=
 \{u\in\mathbb S^{n-1}:\arccos\langle u,v\rangle\le\alpha\}.
\]
Its normalized measure is
\begin{equation}\label{eq:capmeasure}
 A_n(\alpha)
 =\sigma(C(v,\alpha))
 =c_n\int_0^\alpha \sin^{n-2}t\,dt.
\end{equation}
The function $A_n$ is continuous and strictly increasing from $[0,\pi]$
onto $[0,1]$.  We denote its inverse by
\[
 \alpha_n(s)=A_n^{-1}(s),\qquad 0\le s\le1.
\]
Define the cap profile
\begin{equation}\label{eq:Phi}
 \Phi_n(s)
 =
 \frac{c_n}{n-1}\sin^{n-1}\alpha_n(s).
\end{equation}

\begin{lemma}[Concavity of the cap profile]\label{lem:concavity}
The function $\Phi_n$ is continuous and strictly concave on $[0,1]$.  On $(0,1)$,
\begin{equation}\label{eq:PhiDerivative}
 \Phi_n'(s)=\cos\alpha_n(s).
\end{equation}
\end{lemma}

\begin{proof}
From \eqref{eq:capmeasure},
\[
 A_n'(\alpha)=c_n\sin^{n-2}\alpha.
\]
On the other hand,
\[
 \frac{d}{d\alpha}
 \left(
 \frac{c_n}{n-1}\sin^{n-1}\alpha
 \right)
 =
 c_n\sin^{n-2}\alpha\cos\alpha.
\]
The chain rule therefore gives \eqref{eq:PhiDerivative}.  Since
$\alpha_n(s)$ is strictly increasing, $\cos\alpha_n(s)$ is strictly
decreasing.  Thus $\Phi_n'$ is strictly decreasing on $(0,1)$, which,
together with continuity at the endpoints, gives strict concavity on $[0,1]$.
\end{proof}

\begin{remark}\label{rem:n3}
For $n=3$ one has $c_3=1/2$ and
\[
 A_3(\alpha)=\frac{1-\cos\alpha}{2}.
\]
Consequently
\[
 \Phi_3(s)=s(1-s).
\]
This provides a useful consistency check on all constants.
\end{remark}

\section{A sharp spherical trimmed-body inequality}

Let
\[
 F:\mathbb S^{n-1}\longrightarrow\overline{\mathbb B^n}
\]
be measurable.  For $0\le s\le1$, set
\[
 \mathcal A_s
 =
 \left\{
 a\in L^\infty(\mathbb S^{n-1},\sigma):
 0\le a\le1\ \text{a.e.},\quad
 \int_{\mathbb S^{n-1}}a\,d\sigma=s
 \right\},
\]
and define
\begin{equation}\label{eq:trimmed}
 \mathcal Z_s(F)
 =
 \left\{
 \int_{\mathbb S^{n-1}}a(\eta)F(\eta)\,d\sigma(\eta):
 a\in\mathcal A_s
 \right\}.
\end{equation}
The set $\mathcal A_s$ is weak-$*$ compact in $L^\infty$, and the map
\[
 a\longmapsto\int_{\mathbb S^{n-1}}a(\eta)F(\eta)\,d\sigma(\eta)
\]
is weak-$*$ continuous coordinatewise because $F\in L^1$.  Hence
$\mathcal Z_s(F)$ is a compact convex subset of $\mathbb R^n$.

For $u\in\mathbb S^{n-1}$ put
\begin{equation}\label{eq:Mu}
 M_u(s)=h_{\mathcal Z_s(F)}(u)
 =
 \sup_{\substack{0\le a\le1\\\int_{\mathbb S^{n-1}} a\,d\sigma=s}}
 \int_{\mathbb S^{n-1}}
 a(\eta)\langle F(\eta),u\rangle\,d\sigma(\eta).
\end{equation}

The sets \eqref{eq:trimmed} belong naturally to the family of convex
regions generated by zonoid trimming and related average-quantile or
metronoid constructions.  For the zonoid-trimming viewpoint see
\cite{KoshevoyMosler1997,KoshevoyMosler1998,Mosler}; for a modern framework
connecting such constructions with metronoids and sublinear expectations,
see Molchanov and Turin \cite{MolchanovTurin}.  The next estimate is the
spherical extremal inequality needed in the sequel.

\begin{lemma}[Measurable maximizing selectors]\label{lem:quantile-selector}
Fix $s\in(0,1)$.  There exists a jointly measurable function
\[
 (u,\eta)\longmapsto a_{u,s}(\eta)\in[0,1]
\]
such that
\[
 \int_{\mathbb S^{n-1}} a_{u,s}(\eta)\,d\sigma(\eta)=s
\]
and
\[
 M_u(s)
 =
 \int_{\mathbb S^{n-1}} a_{u,s}(\eta)\langle F(\eta),u\rangle\,d\sigma(\eta)
\]
for every $u$ outside a null set.
\end{lemma}

\begin{proof}
Set $X(u,\eta)=\langle F(\eta),u\rangle$ and
\[
 m(u,\lambda)=
 \sigma\{\eta:X(u,\eta)>\lambda\}.
\]
The function $(u,\lambda)\mapsto m(u,\lambda)$ is measurable by Fubini.
Define
\[
 \tau_s(u)=
 \inf\{\lambda\in\mathbb Q:m(u,\lambda)\le s\}.
\]
Since $|X|\le1$ and $0<s<1$, the defining set is nonempty and
$\tau_s(u)\in[-1,1]$.  The map $\tau_s$ is measurable and
\[
 \sigma\{X(u,\cdot)>\tau_s(u)\}
 \le s
 \le
 \sigma\{X(u,\cdot)\ge\tau_s(u)\}.
\]
Set
\[
 b_s(u)
 =
 \sigma\{X(u,\cdot)>\tau_s(u)\},
 \qquad
 d_s(u)
 =
 \sigma\{X(u,\cdot)=\tau_s(u)\}.
\]
Both functions are measurable.  Define
\[
 \theta_s(u)
 =
 \begin{cases}
 \dfrac{s-b_s(u)}{d_s(u)},& d_s(u)>0,\\[1ex]
 0,& d_s(u)=0.
 \end{cases}
\]
The quantile inequalities above give $0\le\theta_s(u)\le1$.  Therefore
\[
 a_{u,s}(\eta)
 =
 \mathbf1_{\{X(u,\eta)>\tau_s(u)\}}
 +
 \theta_s(u)\mathbf1_{\{X(u,\eta)=\tau_s(u)\}}
\]
is jointly measurable and satisfies
\[
 \int_{\mathbb S^{n-1}} a_{u,s}(\eta)\,d\sigma(\eta)=s.
\]
The usual threshold comparison
\[
 \int_{\mathbb S^{n-1}} (a-a_{u,s})(X-\tau_s(u))\,d\sigma\le0
\]
for every competitor $a$ of mass $s$ shows that $a_{u,s}$ is a bathtub
maximizer.
\end{proof}

\begin{theorem}[Sharp trimmed-body mean-width inequality]\label{thm:trimmed}
For every measurable
$F:\mathbb S^{n-1}\to\overline{\mathbb B^n}$ and every $0\le s\le1$,
\begin{equation}\label{eq:trimmedbound}
 \int_{\mathbb S^{n-1}}M_u(s)\,d\sigma(u)\le \Phi_n(s).
\end{equation}
Equivalently,
\[
 w(\mathcal Z_s(F))\le 2\Phi_n(s).
\]
The estimate is sharp.
\end{theorem}

\begin{proof}
At $s=0$ and $s=1$, the trimmed body is a singleton and both sides of
\eqref{eq:trimmedbound} are zero.  Fix $s\in(0,1)$ and use
Lemma~\ref{lem:quantile-selector} to choose jointly measurable maximizers
$a_u=a_{u,s}$.  They satisfy
\begin{equation}\label{eq:selector}
 0\le a_u(\eta)\le1,\qquad
 \int_{\mathbb S^{n-1}} a_u(\eta)\,d\sigma(\eta)=s,
\end{equation}
and
\[
 M_u(s)=\int_{\mathbb S^{n-1}} a_u(\eta)\langle F(\eta),u\rangle\,d\sigma(\eta).
\]
Because the integrand is bounded, Fubini's theorem gives
\begin{align}
 \int_{\mathbb S^{n-1}}M_u(s)\,d\sigma(u)
 &=
 \int_{\mathbb S^{n-1}}
 \left[
 \int_{\mathbb S^{n-1}}
 a_u(\eta)\langle F(\eta),u\rangle\,d\sigma(u)
 \right]d\sigma(\eta).
 \label{eq:FubiniTrim}
\end{align}
Fix $\eta$.  Write
\[
 F(\eta)=\rho(\eta)v(\eta),
 \qquad 0\le\rho(\eta)\le1,
\]
where $v(\eta)\in\mathbb S^{n-1}$ if $\rho(\eta)>0$, and choose
$v(\eta)$ arbitrarily if $F(\eta)=0$.  Define
\begin{equation}\label{eq:qeta}
 q(\eta)=\int_{\mathbb S^{n-1}}a_u(\eta)\,d\sigma(u).
\end{equation}

We now maximize, for fixed $v\in\mathbb S^{n-1}$,
\[
 \int_{\mathbb S^{n-1}}b(u)\langle u,v\rangle\,d\sigma(u)
\]
under $0\le b\le1$ and $\int_{\mathbb S^{n-1}} b\,d\sigma=q$.  The superlevel sets of
$u\mapsto\langle u,v\rangle$ are spherical caps centered at $v$.
Applying the bathtub principle once more
\cite[Theorem~1.14]{LiebLoss}, the cap of measure $q$ is an optimizer.  Hence
\begin{align}
 \sup_{\substack{0\le b\le1\\\int_{\mathbb S^{n-1}} b\,d\sigma=q}}
 \int_{\mathbb S^{n-1}} b(u)\langle u,v\rangle\,d\sigma(u)
 &=
 c_n\int_0^{\alpha_n(q)}
 \cos t\,\sin^{n-2}t\,dt \notag\\
 &=
 \frac{c_n}{n-1}
 \sin^{n-1}\alpha_n(q)
 =\Phi_n(q).
 \label{eq:capextremal}
\end{align}
Consequently,
\[
 \int_{\mathbb S^{n-1}} a_u(\eta)\langle F(\eta),u\rangle\,d\sigma(u)
 \le
 \rho(\eta)\Phi_n(q(\eta))
 \le
 \Phi_n(q(\eta)).
\]
Using \eqref{eq:FubiniTrim},
\begin{equation}\label{eq:preJensen}
 \int_{\mathbb S^{n-1}} M_u(s)\,d\sigma(u)
 \le
 \int_{\mathbb S^{n-1}} \Phi_n(q(\eta))\,d\sigma(\eta).
\end{equation}
Moreover, by Fubini and \eqref{eq:selector},
\begin{align}
 \int_{\mathbb S^{n-1}} q(\eta)\,d\sigma(\eta)
 &=
 \int_{\mathbb S^{n-1}}\!\!\int_{\mathbb S^{n-1}} a_u(\eta)\,d\sigma(u)d\sigma(\eta) \notag\\
 &=
 \int_{\mathbb S^{n-1}}
 \left(\int_{\mathbb S^{n-1}} a_u(\eta)\,d\sigma(\eta)\right)d\sigma(u)
 =s.
 \label{eq:qmean}
\end{align}
Lemma~\ref{lem:concavity} and Jensen's inequality now imply
\[
 \int_{\mathbb S^{n-1}} \Phi_n(q(\eta))\,d\sigma(\eta)
 \le
 \Phi_n\!\left(\int_{\mathbb S^{n-1}} q(\eta)\,d\sigma(\eta)\right)
 =
 \Phi_n(s).
\]
This proves \eqref{eq:trimmedbound}.

To see sharpness, take $F(\eta)=\eta$.  By rotational symmetry,
$\mathcal Z_s(F)$ is a Euclidean ball.  Its support function in any direction
is exactly the cap extremal value \eqref{eq:capextremal}, namely
$\Phi_n(s)$.  Hence equality holds in \eqref{eq:trimmedbound}.
\end{proof}

\begin{remark}[Probability-space form]\label{rem:probability-trimmed}
The source sphere plays no geometric role in the proof of
Theorem~\ref{thm:trimmed}; only its probability measure is used.  Hence the
same argument applies verbatim on an arbitrary probability space
$(\Omega,\mathcal F,\mathbb P)$.  More precisely, if
$Y:\Omega\to\overline{\mathbb B^n}$ is measurable and
\[
 \mathcal Z_s^\Omega(Y)
 =
 \left\{
 \mathbb E[aY]:
 0\le a\le1,\quad \mathbb E a=s
 \right\},
\]
then
\[
 w(\mathcal Z_s^\Omega(Y))\le2\Phi_n(s).
\]
This formulation will be used in the application to zonoid depth.
\end{remark}

\begin{remark}[The planar member of the spherical family]\label{rem:planar}
For $n=2$ one has
\[
 c_2=\frac1\pi,\qquad
 A_2(\alpha)=\frac{\alpha}{\pi},\qquad
 \Phi_2(s)=\frac{\sin(\pi s)}{\pi}.
\]
Thus Theorem~\ref{thm:trimmed} gives precisely the planar sine profile.
Moreover, Cauchy's perimeter formula yields
$\operatorname{Per}(K)=\pi w(K)$ for compact planar convex sets.  Hence
the planar trimmed-range and perimeter mechanisms are exactly the
two-dimensional specialization of the spherical mean-width theory.
\end{remark}

\section{Monotone zonal kernels}

We now prove Theorems~\ref{thm:zonal-main} and
\ref{thm:zonal-rigidity}.  The Poisson theorem will then follow as a
special case.

\subsection{The sharp contraction theorem}

Let $k:[0,\pi]\to\mathbb R$ be continuous and nonincreasing.  Associate
with $k$ the finite positive Lebesgue--Stieltjes measure $\nu_k$ defined by
\[
 \nu_k((a,b])=k(a)-k(b),
 \qquad 0\le a<b\le\pi.
\]
Here $\nu_k(\{0\})=0$.  Since $k$ is continuous, $\nu_k$ has no atoms.  The following identity is
the Lebesgue--Stieltjes form of the standard layer-cake representation;
compare \cite[Theorem~1.13]{LiebLoss}:
\begin{equation}\label{eq:zonal-layercake}
 k(\theta)
 =
 k(\pi)
 +
 \int_{(0,\pi)}
 \mathbf1_{\{\theta\le t\}}\,d\nu_k(t),
 \qquad 0\le\theta\le\pi.
\end{equation}
If $k$ is absolutely continuous, then
$d\nu_k(t)=-k'(t)\,dt$, so \eqref{eq:zonal-layercake} reduces to
\[
k(\theta)=k(\pi)+\int_\theta^\pi (-k'(t))\,dt.
\]
For measurable $F:\mathbb S^{n-1}\to\overline{\mathbb B^n}$ set
\[
 G_k(\xi)=T_kF(\xi)
 =
 \int_{\mathbb S^{n-1}}
 k(d(\xi,\eta))F(\eta)\,d\sigma(\eta)
\]
and
\[
 K_k=\operatorname{co}G_k(\mathbb S^{n-1}).
\]
Because $k$ is continuous and $F$ is bounded, $G_k$ is continuous and
$K_k$ is compact and convex.

\begin{proof}[Proof of Theorem~\ref{thm:zonal-main}]
Fix $u\in\mathbb S^{n-1}$ and write
\[
 X_u(\eta)=\langle F(\eta),u\rangle.
\]
Using \eqref{eq:zonal-layercake},
\begin{align*}
 h_{K_k}(u)
 &=
 \sup_{\xi\in\mathbb S^{n-1}}
 \int_{\mathbb S^{n-1}} k(d(\xi,\eta))X_u(\eta)\,d\sigma(\eta)\\
 &\le
 k(\pi)\int_{\mathbb S^{n-1}} X_u(\eta)\,d\sigma(\eta)
 +
 \int_{(0,\pi)}
 M_u(A_n(t))\,d\nu_k(t).
\end{align*}
Indeed, the cap
\[
 C(\xi,t)=\{\eta:d(\xi,\eta)\le t\}
\]
has normalized measure $A_n(t)$, so
\[
 \int_{C(\xi,t)}X_u\,d\sigma
 \le M_u(A_n(t)).
\]

Integrating in $u$, the constant term vanishes:
\[
 \int_{\mathbb S^{n-1}}
 \left\langle
 \int_{\mathbb S^{n-1}}F\,d\sigma,u
 \right\rangle d\sigma(u)=0.
\]
The trimmed-body theorem therefore gives
\begin{align}
 \int_{\mathbb S^{n-1}}h_{K_k}(u)\,d\sigma(u)
 &\le
 \int_{(0,\pi)}
 \Phi_n(A_n(t))\,d\nu_k(t)\notag\\
 &=
 \frac{c_n}{n-1}
 \int_{(0,\pi)}
 \sin^{n-1}t\,d\nu_k(t).
\label{eq:zonal-before-ibp}
\end{align}
Put
\[
 g_n(t)=\frac{c_n}{n-1}\sin^{n-1}t.
\]
Since $g_n(0)=g_n(\pi)=0$, Stieltjes integration by parts gives
\begin{align}
 \int_{(0,\pi)}g_n(t)\,d\nu_k(t)
 &=
 \int_0^\pi k(t)g_n'(t)\,dt\notag\\
 &=
 c_n\int_0^\pi
 k(t)\cos t\,\sin^{n-2}t\,dt
 =
 \lambda_1(k).
\label{eq:zonal-lambda}
\end{align}
Consequently
\[
 w(K_k)
 =
 2\int_{\mathbb S^{n-1}} h_{K_k}(u)\,d\sigma(u)
 \le2\lambda_1(k).
\]

By Lemma~\ref{lem:degree-one-multiplier}, taking
\[
 F(\eta)=Q\eta,\qquad Q\in O(n),
\]
gives
\[
 T_kF(\xi)=\lambda_1(k)Q\xi.
\]
Also \eqref{eq:zonal-before-ibp}--\eqref{eq:zonal-lambda} show
$\lambda_1(k)\ge0$.  Hence
\[
 K_k=\lambda_1(k)Q\overline{\mathbb B^n}
\]
and equality holds.  This proves sharpness.
\end{proof}

\subsection{Equality structure of the trimmed-body theorem}

We record continuity in the trimming parameter.  Together with the
measurable selectors from Lemma~\ref{lem:quantile-selector}, this will
allow us to trace equality in the contraction estimate.

\begin{lemma}[Continuity in the trimming parameter]\label{lem:M-continuous}
For fixed measurable $F$ with $|F|\le1$ and fixed
$u\in\mathbb S^{n-1}$, the function $s\mapsto M_u(s)$ is
$1$-Lipschitz on $[0,1]$.  Consequently
\[
 J(s):=\int_{\mathbb S^{n-1}}M_u(s)\,d\sigma(u)
\]
is continuous.
\end{lemma}

\begin{proof}
Write $X_u(\eta)=\langle F(\eta),u\rangle$, so that $|X_u|\le1$.
Fix $0\le s<t\le1$.  Since the admissible sets are weak-$*$ compact and
the objective is weak-$*$ continuous, maximizers exist.  Let $a_t$ be a
maximizer of mass $t$.  The function
\[
 b=\frac{s}{t}a_t
\]
has mass $s$ (with $t>0$ because $s<t$), and hence
\[
 M_u(s)
 \ge \int_{\mathbb S^{n-1}} b(\eta)X_u(\eta)\,d\sigma(\eta)
 \ge M_u(t)-(t-s),
\]
because $\int_{\mathbb S^{n-1}}(a_t-b)\,d\sigma=t-s$ and $|X_u|\le1$.
Conversely, if $a_s$ is a maximizer of mass $s$, then
\[
 c=a_s+\frac{t-s}{1-s}(1-a_s)
\]
is admissible of mass $t$ (the case $s=1$ cannot occur when $s<t$), and
\[
 M_u(t)
 \ge \int_{\mathbb S^{n-1}} c(\eta)X_u(\eta)\,d\sigma(\eta)
 \ge M_u(s)-(t-s).
\]
Thus $|M_u(t)-M_u(s)|\le t-s$.  Integration in $u$ gives the continuity of
$J$.
\end{proof}

\subsection{Rigidity for strictly decreasing kernels}

\begin{proof}[Proof of Theorem~\ref{thm:zonal-rigidity}]
The converse direction has already been proved in the sharpness part of
Theorem~\ref{thm:zonal-main}.  Suppose therefore that
\begin{equation}\label{eq:zonal-rigidity-equality}
 w(K_k)=2\lambda_1(k).
\end{equation}
Because $k$ is strictly decreasing, the Stieltjes measure $\nu_k$ has full
support:
\[
 \nu_k((a,b))>0
 \qquad
 (0\le a<b\le\pi).
\]

Let
\[
 J(s)=\int_{\mathbb S^{n-1}} M_u(s)\,d\sigma(u).
\]
The proof of Theorem~\ref{thm:zonal-main} gives
\[
 \int_{\mathbb S^{n-1}} h_{K_k}(u)\,d\sigma(u)
 \le
 \int_{(0,\pi)}J(A_n(t))\,d\nu_k(t)
 \le
 \int_{(0,\pi)}\Phi_n(A_n(t))\,d\nu_k(t)
 =
 \lambda_1(k).
\]
By \eqref{eq:zonal-rigidity-equality}, both inequalities are equalities.
The function
\[
 t\longmapsto
 \Phi_n(A_n(t))-J(A_n(t))
\]
is continuous and nonnegative by Lemma~\ref{lem:M-continuous}.  Since
$\nu_k$ has full support, its integral can vanish only if
\begin{equation}\label{eq:J-equality-all}
 J(s)=\Phi_n(s),
 \qquad 0\le s\le1.
\end{equation}

Fix $s\in(0,1)$ and choose the maximizing selectors of
Lemma~\ref{lem:quantile-selector}.  Define
\[
 q_s(\eta)
 =
 \int_{\mathbb S^{n-1}}a_{u,s}(\eta)\,d\sigma(u).
\]
Then $\int_{\mathbb S^{n-1}} q_s(\eta)\,d\sigma(\eta)=s$.  Retaining all inequalities in the proof of
Theorem~\ref{thm:trimmed}, and using \eqref{eq:J-equality-all}, gives
\begin{align}
 \Phi_n(s)
 &=
 J(s)\notag\\
 &\le
 \int_{\mathbb S^{n-1}} |F(\eta)|\Phi_n(q_s(\eta))\,d\sigma(\eta)\notag\\
 &\le
 \int_{\mathbb S^{n-1}} \Phi_n(q_s(\eta))\,d\sigma(\eta)\notag\\
 &\le
 \Phi_n\!\left(\int_{\mathbb S^{n-1}} q_s(\eta)\,d\sigma(\eta)\right)
 =
 \Phi_n(s).
\label{eq:zonal-rigidity-chain}
\end{align}
Thus equality holds at every stage.  Since
\[
 \Phi_n'(s)=\cos\alpha_n(s)
\]
is strictly decreasing on $(0,1)$, $\Phi_n$ is strictly concave on $[0,1]$.
Equality in Jensen's inequality yields
\begin{equation}\label{eq:zonal-q}
 q_s(\eta)=s
 \qquad\text{for a.e. }\eta.
\end{equation}
As $\Phi_n(s)>0$, equality in the preceding inequality also gives
\begin{equation}\label{eq:zonal-F-sphere}
 |F(\eta)|=1
 \qquad\text{for a.e. }\eta.
\end{equation}

Put
\[
 c_s=\cos\alpha_n(s).
\]
For fixed $v\in\mathbb S^{n-1}$, the unique bathtub maximizer of mass $s$
for the function $u\mapsto\langle u,v\rangle$ is, up to a null set,
\[
 \mathbf1_{\{\langle u,v\rangle>c_s\}}.
\]
Indeed, if $b_*$ denotes this indicator and $b$ is any competitor of mass
$s$, then
\[
 \int_{\mathbb S^{n-1}}(b_*(u)-b(u))(\langle u,v\rangle-c_s)\,d\sigma(u)\ge0,
\]
and equality forces $b=b_*$ away from the zero-measure level set
$\{\langle u,v\rangle=c_s\}$.

Equality in the first step of \eqref{eq:zonal-rigidity-chain}, together
with \eqref{eq:zonal-q} and \eqref{eq:zonal-F-sphere}, therefore gives
\begin{equation}\label{eq:zonal-selector-cap}
 a_{u,s}(\eta)
 =
 \mathbf1_{\{\langle F(\eta),u\rangle>c_s\}}
 \qquad\text{for a.e. }(u,\eta).
\end{equation}
Choose a countable dense set $S_0\subset(0,1)$.  Since every selector has
mass $s$, Fubini and \eqref{eq:zonal-selector-cap} imply that for almost
every $u$ and every $s\in S_0$,
\begin{equation}\label{eq:zonal-tail-dense}
 \sigma\{
 \eta:\langle F(\eta),u\rangle>c_s
 \}
 =
 s.
\end{equation}
The numbers $c_s$, $s\in S_0$, are dense in $(-1,1)$.  For the model
variable $Z(\eta)=\eta_n$ one has
\[
 \sigma\{Z>c_s\}=s,
 \qquad 0<s<1,
\]
by the definition of $c_s=\cos\alpha_n(s)$.  Hence, for almost every $u$,
\eqref{eq:zonal-tail-dense} identifies the tail function
$c\mapsto\sigma\{\langle F(\eta),u\rangle>c\}$ with the tail function of
$Z$ on the dense set $\{c_s:s\in S_0\}$.  Both tail functions are monotone
and right-continuous, while the spherical-coordinate distribution is
continuous.  It follows that the two tail functions agree for every
$c\in(-1,1)$, and therefore
\begin{equation}\label{eq:zonal-projection-law}
 \langle F(\eta),u\rangle
 \quad\text{has the same distribution as}\quad
 \eta_n.
\end{equation}
In particular this distribution has no atoms.

We next return to the first inequality in the proof of
Theorem~\ref{thm:zonal-main}.  Define
\[
 B_k(u)
 =
 k(\pi)\int_{\mathbb S^{n-1}}X_u\,d\sigma
 +
 \int_{(0,\pi)}M_u(A_n(t))\,d\nu_k(t).
\]
Then
\[
 h_{K_k}(u)\le B_k(u)
 \qquad\text{for every }u.
\]
Because \eqref{eq:J-equality-all} holds, integration in $u$ gives
\[
 \int_{\mathbb S^{n-1}}B_k(u)\,d\sigma(u)
 =
 \lambda_1(k)
 =
 \int_{\mathbb S^{n-1}}h_{K_k}(u)\,d\sigma(u).
\]
Hence the nonnegative deficit $B_k-h_{K_k}$ vanishes for almost every
$u$.

Fix such a direction $u$ for which
\eqref{eq:zonal-projection-law} also holds.  Choose
$\xi_u\in\mathbb S^{n-1}$ at which
\[
 \xi\longmapsto
 \int_{\mathbb S^{n-1}} k(d(\xi,\eta))X_u(\eta)\,d\sigma(\eta)
\]
attains its maximum.  Applying the layer-cake identity at this maximizing
direction yields
\[
 B_k(u)-h_{K_k}(u)
 =
 \int_{(0,\pi)}D_u(t)\,d\nu_k(t),
\]
where
\[
 D_u(t)
 =
 M_u(A_n(t))
 -
 \int_{C(\xi_u,t)}X_u(\eta)\,d\sigma(\eta)
 \ge0.
\]
Thus the integral of $D_u$ is zero.  The function $D_u$ is continuous:
the first term is continuous by Lemma~\ref{lem:M-continuous}, while the
second is continuous because spherical caps vary continuously in measure
and $X_u$ is bounded.  Since $\nu_k$ has full support,
\[
 D_u(t)=0,\qquad0<t<\pi.
\]
Thus every cap $C(\xi_u,t)$ is a bathtub optimizer of mass $A_n(t)$ for
$X_u$.  By \eqref{eq:zonal-projection-law},
\[
 \sigma\{X_u>\cos t\}=A_n(t),
\]
and the level set $\{X_u=\cos t\}$ has measure zero.  Hence the unique
optimizer of this mass is $\{X_u>\cos t\}$, and therefore, for every
$t\in(0,\pi)$,
\begin{equation}\label{eq:zonal-superlevel}
 \mathbf1_{\{X_u(\eta)>\cos t\}}
 =
 \mathbf1_{\{\langle\eta,\xi_u\rangle>\cos t\}}
 \qquad\text{for a.e. }\eta.
\end{equation}

For each fixed $t$, the two indicators in
\eqref{eq:zonal-superlevel} agree for almost every $\eta$.  Integrating
their absolute difference over $(0,\pi)\times\mathbb S^{n-1}$ and applying
Tonelli's theorem therefore shows that, for almost every $\eta$, they agree
for almost every $t$.  If two numbers $a,b\in[-1,1]$ satisfy
\[
 \mathbf1_{\{a>\cos t\}}=\mathbf1_{\{b>\cos t\}}
 \quad\text{for a.e. }t\in(0,\pi),
\]
then $a=b$, since otherwise the two indicators differ on a nonempty interval
of $t$-values.  Consequently
\[
 X_u(\eta)=\langle\eta,\xi_u\rangle
 \qquad\text{for a.e. }\eta.
\]
Thus
\[
 \langle F(\cdot),u\rangle
 \in\mathcal H_1(\mathbb S^{n-1})
\]
for almost every $u$.

Finally set
\[
 E=
 \{u\in\mathbb R^n:
 \langle F(\cdot),u\rangle\in\mathcal H_1(\mathbb S^{n-1})\}.
\]
This is a linear subspace of $\mathbb R^n$, and
$E\cap\mathbb S^{n-1}$ has full spherical measure.  Hence $E=\mathbb R^n$.
Therefore every coordinate of $F$ is a first spherical harmonic, so
\[
 F(\eta)=A\eta
 \qquad\text{a.e.}
\]
for some real matrix $A$.  By \eqref{eq:zonal-F-sphere},
$|A\eta|=1$ for almost every $\eta\in\mathbb S^{n-1}$.  Continuity then
gives
\[
 \eta^T(A^TA-I)\eta=0
 \qquad\text{for every }\eta\in\mathbb S^{n-1},
\]
hence $A^TA=I$.  Thus $A=Q\in O(n)$, completing the proof.
\end{proof}

\subsection{A semigroup consequence}

We use the standard terminology of Markov semigroups; see, for example,
\cite[Chapter~1]{BakryGentilLedoux}.  In the present subsection, a Markov
semigroup $(T_t)_{t\ge0}$ is a family of linear operators satisfying the
semigroup law
\[
 T_0=I,\qquad T_{t+s}=T_tT_s,
\]
which preserves positivity and constants.  Rotational invariance means that
$T_t(\varphi\circ Q)=(T_t\varphi)\circ Q$ for every $Q\in O(n)$.  No general
semigroup theory beyond these properties and the spectral assumption below
is needed.

\begin{corollary}[Zonal semigroup contraction]\label{cor:zonal-semigroup}
Let $(T_t)_{t\ge0}$ be a rotationally invariant Markov semigroup on
$\mathbb S^{n-1}$.  Assume that for each $t>0$ it has a continuous
nonincreasing zonal kernel $k_t(d(\xi,\eta))$, and suppose that
\[
 T_tY=e^{-\gamma_1t}Y
 \qquad
 \text{for every }Y\in\mathcal H_1(\mathbb S^{n-1}).
\]
Then for every $t>0$ and every measurable
$F:\mathbb S^{n-1}\to\overline{\mathbb B^n}$,
\[
 w\!\left(
 \operatorname{co}T_tF(\mathbb S^{n-1})
 \right)
 \le
 2e^{-\gamma_1t}.
\]
Moreover, for $1\le j\le n$,
\[
 V_j\!\left(
 \operatorname{co}T_tF(\mathbb S^{n-1})
 \right)
 \le
 e^{-j\gamma_1t}V_j(\overline{\mathbb B^n}).
\]
If $k_t$ is strictly decreasing and equality holds in the mean-width
estimate, then
\[
 F(\eta)=Q\eta
 \qquad\text{a.e.}
\]
for some $Q\in O(n)$.
\end{corollary}

\begin{proof}
The degree-one multiplier is
$\lambda_1(k_t)=e^{-\gamma_1t}$.  Apply
Theorems~\ref{thm:zonal-main} and \ref{thm:zonal-rigidity}; the
intrinsic-volume estimate follows from
\eqref{eq:extended-Urysohn}.
\end{proof}

\subsection{Necessity of radial monotonicity}

Theorem~\ref{thm:zonal-main} cannot be extended to arbitrary positive
rotationally invariant Markov kernels using only their degree-one
multiplier.

\begin{proposition}[Failure without radial monotonicity]
\label{prop:nonmonotone}
For every $n\ge2$ there exists a continuous nonnegative zonal Markov kernel
$k$ whose degree-one multiplier is zero, and a measurable
$F:\mathbb S^{n-1}\to\overline{\mathbb B^n}$, such that
\[
 w\!\left(
 \operatorname{co}T_kF(\mathbb S^{n-1})
 \right)>0.
\]
\end{proposition}

\begin{proof}
Put
\[
 q(t)=\cos^2t-\frac1n
\]
and, for $0<\varepsilon\le1$, set
\[
 k_\varepsilon(t)=1+\varepsilon q(t).
\]
Since $q(t)\ge-1/n$, the kernel is nonnegative.  Moreover,
\[
 \int_{\mathbb S^{n-1}}
 \left(
 \langle\xi,\eta\rangle^2-\frac1n
 \right)d\sigma(\eta)=0,
\]
so $k_\varepsilon$ is Markov.  Its degree-one multiplier is zero because
\[
 \int_{\mathbb S^{n-1}}
 \left(
 \langle\xi,\eta\rangle^2-\frac1n
 \right)\eta\,d\sigma(\eta)=0
\]
by oddness, and the constant part also annihilates degree-one functions.

Let
\[
 Y(\eta)=\eta_1^2-\frac1n,
 \qquad
 F(\eta)=\frac{Y(\eta)}{\|Y\|_\infty}e_1.
\]
Then $|F|\le1$, while at $\xi=e_1$,
\[
 T_{k_\varepsilon}F(e_1)
 =
 \frac{\varepsilon}{\|Y\|_\infty}
 \left(
 \int_{\mathbb S^{n-1}}Y(\eta)^2\,d\sigma(\eta)
 \right)e_1
 \ne0.
\]
On the other hand, Fubini's theorem and the symmetry of the zonal kernel
give
\begin{align*}
 \int_{\mathbb S^{n-1}}T_{k_\varepsilon}F(\xi)\,d\sigma(\xi)
 &=
 \int_{\mathbb S^{n-1}}F(\eta)
 \left(
   \int_{\mathbb S^{n-1}}
   k_\varepsilon(d(\xi,\eta))\,d\sigma(\xi)
 \right)d\sigma(\eta)\\
 &=\int_{\mathbb S^{n-1}}F(\eta)\,d\sigma(\eta)=0,
\end{align*}
where the inner integral equals $1$ by the Markov normalization.
Hence $T_{k_\varepsilon}F$ is not constant, so its convex hull is
nontrivial and has positive mean width.  Since the degree-one multiplier
is zero, no estimate of the form
$w\le2\lambda_1$ can hold for this kernel.
\end{proof}

\section{The Poisson specialization}

For $0<\rho<1$, the Poisson kernel of $\mathbb B^n$ with respect to
normalized spherical measure is
\[
 P_\rho(\xi,\eta)
 =
 p_\rho(d(\xi,\eta)),
\]
where
\begin{equation}\label{eq:prho}
 p_\rho(t)
 =
 \frac{1-\rho^2}
 {(1-2\rho\cos t+\rho^2)^{n/2}}.
\end{equation}
It is continuous and strictly decreasing on $(0,\pi)$.

\begin{corollary}[Poisson mean-width contraction and rigidity]
\label{cor:PoissonWidth}
Let
$F:\mathbb S^{n-1}\to\overline{\mathbb B^n}$ be measurable and define
\[
 G_\rho(\xi)
 =
 \int_{\mathbb S^{n-1}}
 P_\rho(\xi,\eta)F(\eta)\,d\sigma(\eta),
 \qquad
 K_\rho=\operatorname{co}G_\rho(\mathbb S^{n-1}).
\]
Then
\[
w(K_\rho)\le2\rho.
\]
Equality holds if and only if
\[
 F(\eta)=Q\eta
 \qquad\text{a.e.}
\]
for some $Q\in O(n)$.
\end{corollary}

\begin{proof}
The degree-one multiplier of the Poisson operator is $\rho$.  Indeed, for
the coordinate function $\eta\mapsto\eta_n$,
\[
 \int_{\mathbb S^{n-1}}
 P_\rho(e_n,\eta)\eta_n\,d\sigma(\eta)
 =
 \rho
\]
by Poisson reproduction of the harmonic function $x\mapsto x_n$.  Hence
\[
 \lambda_1(p_\rho)=\rho.
\]
The result follows from Theorems~\ref{thm:zonal-main} and
\ref{thm:zonal-rigidity}.
\end{proof}

\section{Harmonic self-maps and one-radius rigidity}

We now pass from boundary data to arbitrary harmonic self-maps.

\begin{proof}[Proof of Theorem~\ref{thm:main}]
Fix $0<r<R<1$ and put
\[
 \rho=\frac rR,\qquad
 F_R(\eta)=f(R\eta),\qquad \eta\in\mathbb S^{n-1}.
\]
Since $f(\mathbb B^n)\subset\mathbb B^n$,
\[
 |F_R(\eta)|<1.
\]
Define $g(x)=f(Rx)$.  The map $g$ is harmonic in the ball of radius $1/R>1$,
so the classical Poisson representation on the unit sphere applies without
any appeal to boundary-limit theory.  For $\xi\in\mathbb S^{n-1}$,
\[
 f(r\xi)
 =
 g(\rho\xi)
 =
 \int_{\mathbb S^{n-1}}
 P_\rho(\xi,\eta)F_R(\eta)\,d\sigma(\eta).
\]
Corollary~\ref{cor:PoissonWidth} therefore gives
\[
 w\!\left(\operatorname{co}f(r\mathbb S^{n-1})\right)
 \le
 2\frac rR.
\]
This holds for every $R\in(r,1)$.  Letting $R\uparrow1$ yields
\[
 w\!\left(\operatorname{co}f(r\mathbb S^{n-1})\right)\le2r.
\]
Sharpness follows from $f(x)=Qx$, $Q\in O(n)$.
\end{proof}

For rigidity, we represent $f$ using a single boundary datum, so that
equality at $r_0$ can be tested directly with the kernel $P_{r_0}$.
We include the standard weak-$*$ argument; compare
\cite[Theorem~6.13]{AxlerBourdonRamey}.

\begin{lemma}[Bounded harmonic boundary representation]
\label{lem:boundary-representation}
Let
\[
 f:\mathbb B^n\to\mathbb R^n
\]
be harmonic and bounded by $1$ in norm.  Then there exists
\[
 F\in L^\infty(\mathbb S^{n-1};\mathbb R^n),
 \qquad |F|\le1\ \text{a.e.},
\]
such that
\begin{equation}\label{eq:boundary-representation}
 f(r\xi)
 =
 \int_{\mathbb S^{n-1}}
 P_r(\xi,\eta)F(\eta)\,d\sigma(\eta),
 \qquad 0\le r<1.
\end{equation}
\end{lemma}

\begin{proof}
This is the vector-valued form of the standard boundary representation for
bounded harmonic functions; see \cite[Theorem~6.13]{AxlerBourdonRamey}.
Applying that result componentwise and taking the a.e. radial boundary
limits gives a boundary function $F$ satisfying $|F|\le1$ a.e.
\end{proof}
\begin{proof}[Proof of Theorem~\ref{thm:harmonic-rigidity}]
Let $F$ be the boundary datum given by
Lemma~\ref{lem:boundary-representation}.  At the radius $r_0$,
\[
 f(r_0\xi)
 =
 \int_{\mathbb S^{n-1}}
 P_{r_0}(\xi,\eta)F(\eta)\,d\sigma(\eta).
\]
Hence the hypothesis
\[
 w\!\left(\operatorname{co}f(r_0\mathbb S^{n-1})\right)=2r_0
\]
is exactly the equality case in Corollary~\ref{cor:PoissonWidth}.
Its rigidity statement therefore gives
\[
 F(\eta)=Q\eta
 \qquad\text{for a.e. }\eta\in\mathbb S^{n-1}
\]
for some $Q\in O(n)$.  Using
\eqref{eq:boundary-representation} again and Poisson reproduction,
\[
 f(r\xi)
 =
 Q\int_{\mathbb S^{n-1}}P_r(\xi,\eta)\eta\,d\sigma(\eta)
 =
 rQ\xi,
 \qquad 0\le r<1.
\]
Thus $f(x)=Qx$ on $\mathbb B^n$.  The converse is immediate.
\end{proof}
\begin{lemma}[Harmonic convex-hull property]\label{lem:convexhull}
For every harmonic $f:\mathbb B^n\to\mathbb R^n$ and every $0<r<1$,
\[
 f(r\mathbb B^n)
 \subset
 \operatorname{co}f(r\mathbb S^{n-1}).
\]
\end{lemma}

\begin{proof}
Fix $x\in r\mathbb B^n$.  Applying the Poisson formula in the ball
$r\mathbb B^n$ to each coordinate of $f$ writes $f(x)$ as
\[
 f(x)
 =
 \int_{\mathbb S^{n-1}}
 P_{r,x}(\eta)f(r\eta)\,d\sigma(\eta),
\]
where $P_{r,x}(\eta)\ge0$ and
\[
 \int_{\mathbb S^{n-1}}P_{r,x}(\eta)\,d\sigma(\eta)=1.
\]
Thus $f(x)$ is a convex average of points of $f(r\mathbb S^{n-1})$.
\end{proof}

\begin{proof}[Proof of Corollary~\ref{cor:volume}]
Let
\[
 K_r=\operatorname{co}f(r\mathbb S^{n-1}).
\]
By Lemma~\ref{lem:convexhull},
\[
 f(r\mathbb B^n)\subset K_r.
\]
We also record that $f(r\mathbb B^n)$ is Lebesgue measurable.  Indeed, for
any sufficiently large integer $j_0$,
\[
 r\mathbb B^n
 =
 \bigcup_{j\ge j_0}\overline{B}_{\,r-1/j},
\]
and continuity of $f$ gives
\[
 f(r\mathbb B^n)
 =
 \bigcup_{j\ge j_0}
 f\!\left(\overline{B}_{\,r-1/j}\right).
\]
Each set on the right is compact, so $f(r\mathbb B^n)$ is
$\sigma$-compact and hence Borel measurable.  Therefore
\[
 |f(r\mathbb B^n)|\le|K_r|.
\]
If $K_r$ has empty interior, then $|K_r|=0$ and there is nothing to prove.
Otherwise Urysohn's inequality \eqref{eq:urysohn} and
Theorem~\ref{thm:main} imply
\[
 \left(\frac{|K_r|}{\omega_n}\right)^{1/n}
 \le
 \frac{w(K_r)}2
 \le r.
\]
Therefore
\[
 |f(r\mathbb B^n)|
 \le|K_r|
 \le\omega_n r^n.
\]
For $f(x)=Qx$, $Q\in O(n)$, equality holds for every $r$.
\end{proof}

\begin{proof}[Proof of Corollary~\ref{cor:planar}]
Identify $\mathbb B^2$ with $\mathbb D$ and $\mathbb S^1$ with
$\mathbb T$.  For every compact planar convex set $K$, Cauchy's perimeter
formula and the normalization \eqref{eq:meanwidth} give
\[
 \operatorname{Per}(K)
 =
 \int_0^{2\pi}h_K(\phi)\,d\phi
 =
 \pi w(K),
\]
with the same identity by continuous extension for degenerate convex sets.
Theorem~\ref{thm:main} therefore gives
\[
 \operatorname{Per}\!\left(\operatorname{co}f(r\mathbb T)\right)
 \le2\pi r.
\]
The area estimate is the case $n=2$ of
Corollary~\ref{cor:volume}; equivalently, it follows by the planar
isoperimetric inequality.  Rotations and reflections of the disk give
equality.
\end{proof}

\section{Consequences and geometric applications}

Standard inequalities of convex geometry turn the mean-width bound into
estimates for intrinsic volumes and parallel bodies.

\subsection{The full intrinsic-volume hierarchy}

\begin{proof}[Proof of Corollary~\ref{cor:intrinsic}]
Let
\[
 K_r=\operatorname{co}f(r\mathbb S^{n-1}).
\]
By \eqref{eq:extended-Urysohn}, \eqref{eq:V1-meanwidth}, and
Theorem~\ref{thm:main},
\[
 \left(
 \frac{V_j(K_r)}{V_j(\overline{\mathbb B^n})}
 \right)^{1/j}
 \le
 \frac{w(K_r)}2
 \le r.
\]
Hence
\[
 V_j(K_r)
 \le
 r^jV_j(\overline{\mathbb B^n})
 =
 V_j(r\overline{\mathbb B^n}),
 \qquad j=1,\ldots,n.
\]
For $f(x)=Qx$ with $Q\in O(n)$, one has
$K_r=r\overline{\mathbb B^n}$, so equality holds simultaneously for every
$j$.
\end{proof}

\begin{corollary}[Rigidity of intrinsic-volume contraction]
\label{cor:intrinsic-rigidity}
Let $K_r=\operatorname{co}f(r\mathbb S^{n-1})$.  If for some
$r\in(0,1)$ and some $j\in\{1,\ldots,n\}$,
\[
 V_j(K_r)=V_j(r\overline{\mathbb B^n}),
\]
then
\[
 f(x)=Qx
\]
for some $Q\in O(n)$.
\end{corollary}

\begin{proof}
The extended isoperimetric inequality and
Theorem~\ref{thm:main} give
\[
 r
 =
 \left(
 \frac{V_j(K_r)}
 {V_j(\overline{\mathbb B^n})}
 \right)^{1/j}
 \le
 \frac{w(K_r)}2
 \le r.
\]
Thus $w(K_r)=2r$, and Theorem~\ref{thm:harmonic-rigidity} applies.
\end{proof}

\begin{corollary}[Rigidity of image-volume contraction]
\label{cor:volume-rigidity}
If for some $r\in(0,1)$,
\[
 |f(r\mathbb B^n)|=\omega_nr^n,
\]
then
\[
 f(x)=Qx
\]
for some $Q\in O(n)$.
\end{corollary}

\begin{proof}
With $K_r=\operatorname{co}f(r\mathbb S^{n-1})$,
\[
 |f(r\mathbb B^n)|
 \le
 |K_r|
 \le
 \omega_n\left(\frac{w(K_r)}2\right)^n
 \le
 \omega_nr^n.
\]
Equality at the two endpoints forces $w(K_r)=2r$, and
Theorem~\ref{thm:harmonic-rigidity} applies.
\end{proof}

The case $j=n$ recovers the convex-hull volume estimate.  In dimension two,
$j=1$ is equivalent to the sharp perimeter estimate through
$\operatorname{Per}(K)=\pi w(K)$.  In higher dimensions the intermediate
values of $j$ give genuinely additional geometric information.

In particular, if $f:\mathbb D\to\mathbb D$ is harmonic and equality holds
at some radius either in
\[
 \operatorname{Per}\!\left(\operatorname{co}f(r\mathbb T)\right)\le2\pi r
\]
or in
\[
 |f(r\mathbb D)|\le\pi r^2,
\]
then
\[
 f(z)=e^{i\theta}z
 \qquad\text{or}\qquad
 f(z)=e^{i\theta}\overline z
\]
for some $\theta\in\mathbb R$.

We record three immediate consequences without introducing additional formal
statements.  Let
\[
K_r=\operatorname{co}f(r\mathbb S^{n-1}).
\]

First, if $K_r$ has nonempty interior, then
\[
2V_{n-1}(K_r)=\mathcal H^{n-1}(\partial K_r).
\]
Hence the case $j=n-1$ of the intrinsic-volume contraction gives
\[
\mathcal H^{n-1}(\partial K_r)
\le
2r^{n-1}V_{n-1}(\overline{\mathbb B^n})
=
n\omega_n r^{n-1}.
\]

Second, Steiner's formula \eqref{eq:Steiner} gives
\[
|K_r+t\overline{\mathbb B^n}|
=
\sum_{j=0}^n
\omega_{n-j}V_j(K_r)t^{\,n-j}.
\]
Since
\[
V_j(K_r)
\le
V_j(r\overline{\mathbb B^n})
\qquad (0\le j\le n),
\]
where the case $j=0$ is equality, comparison of the coefficients yields
\[
\begin{aligned}
|K_r+t\overline{\mathbb B^n}|
&\le
\sum_{j=0}^n
\omega_{n-j}
V_j(r\overline{\mathbb B^n})t^{\,n-j} \\
&=
|r\overline{\mathbb B^n}
+t\overline{\mathbb B^n}| \\
&=
\omega_n(r+t)^n ,
\qquad t\ge0.
\end{aligned}
\]
Since
\[
f(r\mathbb B^n)\subset K_r,
\]
we also have
\[
|f(r\mathbb B^n)+t\overline{\mathbb B^n}|
\le
\omega_n(r+t)^n.
\]

Finally, let
\[
E=b+A(\mathbb B^n),
\qquad A\in GL(n,\mathbb R),
\]
and suppose that $f:\mathbb B^n\to E$ is harmonic.  Then
\[
g(x)=A^{-1}(f(x)-b)
\]
is a harmonic self-map of $\mathbb B^n$.  Applying the volume contraction
to $g$ gives
\[
|g(r\mathbb B^n)|\le\omega_n r^n,
\]
and therefore
\[
|f(r\mathbb B^n)|
=
|\det A|\,|g(r\mathbb B^n)|
\le
r^n|E|.
\]
All three estimates are sharp for orthogonal, respectively affine
orthogonal, extremals.
\begin{remark}
The argument is formulated for
$f:\mathbb B^n\to\mathbb B^n$ with $n\ge2$.  The spherical cap profile in
Theorem~\ref{thm:trimmed} is averaged over target directions
$u\in\mathbb S^{n-1}$, while the Poisson layer-cake decomposition produces
caps in the domain sphere $\mathbb S^{n-1}$.  The equality of these
dimensions is what makes
\[
 \Phi_n(A_n(t))
 =
 \frac{c_n}{n-1}\sin^{n-1}t
\]
fit exactly with the first spherical-harmonic reproduction identity.  We do
not claim here an analogous sharp theorem for maps
$\mathbb B^n\to\mathbb B^m$ when $m\ne n$.
\end{remark}

\section{Hyperbolic-harmonic self-maps}\label{hyp:sec:hyperbolic}

Theorems~\ref{thm:zonal-main} and \ref{thm:zonal-rigidity} reduce the
hyperbolic problem to identifying the first multiplier of the
Poisson--Szeg\H{o} kernel.  We compute that multiplier, derive the sharp
global bounds, and then use domain automorphisms to obtain differential
estimates at arbitrary points.

For later use, if $A:\mathbb R^n\to\mathbb R^n$ is linear, then
\begin{equation}\label{eq:hyp-ellipsoid-width}
 \frac12 w(A\overline{\mathbb B^n})
 =\int_{\mathbb S^{n-1}}|A^T\xi|\,d\sigma(\xi).
\end{equation}

\subsection{Boundary representation and hyperbolic automorphisms}

The kernel \eqref{eq:hyp-Ph} satisfies
\begin{equation}\label{hyp:eq:Ph-normalization}
 \int_{\Sph^{n-1}}\Ph(x,\zeta)\,d\sigma(\zeta)=1.
\end{equation}
For $F\in L^1(\Sph^{n-1},\R^m)$ its Poisson--Szeg\H{o} integral is
\begin{equation}\label{hyp:eq:Poisson-extension}
 P_h[F](x)=\int_{\Sph^{n-1}}\Ph(x,\zeta)F(\zeta)\,d\sigma(\zeta).
\end{equation}
Each component of $P_h[F]$ is hyperbolic harmonic.  Conversely, bounded
hyperbolic-harmonic functions admit essentially bounded boundary data;
see Stoll \cite[Theorem 7.1.1(c)]{Stoll2016} and Chen--Kalaj
\cite{ChenKalaj2021}.  Applying this representation componentwise and
taking almost-everywhere boundary limits preserves the bound on the
Euclidean norm.  Thus every hyperbolic-harmonic $u:\B^n\to\B^n$ has
the representation
\begin{equation}\label{hyp:eq:boundary-data}
 u=P_h[F]
 \quad\text{for some }F\in L^\infty(\Sph^{n-1},\R^n)
 \quad\text{with }|F(\zeta)|\le1\text{ a.e.}
\end{equation}

We shall also use that the equation $\Dh v=0$ is invariant under precomposition with hyperbolic automorphisms of the ball. If $\varphi$ is a ball automorphism with $\varphi(0)=a$, then its Euclidean derivative at the origin is conformal and
\begin{equation}\label{hyp:eq:automorphism-derivative}
 D\varphi(0)=(1-|a|^2)O
\end{equation}
for some $O\in O(n)$. These are standard facts for the Poincar\'e ball; see \cite{Stoll2016}.

\subsection{A hyperbolic-harmonic convex-hull maximum principle}

The operator \eqref{eq:hyp-laplacian-intro} is not dilation invariant.
We obtain the required convex-hull inclusion from its maximum principle
on compact subballs.

\begin{lemma}[Hyperbolic-harmonic convex-hull property]\label{hyp:lem:convex-hull}
Let $u:\B^n\to\R^n$ be hyperbolic harmonic. Then, for every $0<r<1$,
\begin{equation}\label{hyp:eq:convex-hull-property}
 u(r\B^n)\subset \co u(r\Sph^{n-1}).
\end{equation}
\end{lemma}

\begin{proof}
Put $K=\co u(r\Sph^{n-1})$. If $u(x_0)\notin K$ for some $|x_0|<r$, the separating-hyperplane theorem gives a vector $v\in\Sph^{n-1}$ such that
\[
 \langle u(x_0),v\rangle>
 \sup_{|x|=r}\langle u(x),v\rangle.
\]
But $x\mapsto\langle u(x),v\rangle$ satisfies $\Dh f=0$. On the closed ball of radius $r$, the operator is uniformly elliptic, so the maximum principle gives a contradiction.
\end{proof}

\subsection{The Poisson--Szeg\H{o} kernel and its first spherical mode}
For $0<r<1$ and $\theta\in[0,\pi]$, put
\begin{equation}\label{hyp:eq:prh-merged}
 p_r^h(\theta)=\left(\frac{1-r^2}{1-2r\cos\theta+r^2}\right)^{n-1}.
\end{equation}
Then
\[
 (p_r^h)'(\theta)
 =-\frac{2(n-1)r(1-r^2)^{n-1}\sin\theta}
 {(1-2r\cos\theta+r^2)^n}<0,
 \qquad 0<\theta<\pi.
\]
Thus $p_r^h$ is a strictly decreasing zonal kernel and Theorems~\ref{thm:zonal-main} and \ref{thm:zonal-rigidity} apply. Its degree-one multiplier is exactly $\Lambda_n(r)$ from \eqref{eq:hyp-Lambda-intro}.

\begin{proposition}[Explicit profile]\label{hyp:prop:Lambda}
For $n\ge2$ and $0\le r<1$,
\begin{equation}\label{hyp:eq:Lambda-explicit}
 \Lambda_n(r)=\frac{2(n-1)}{n}\,r\,
 {}_2F_1\!\left(1,1-\frac n2;\frac{n+2}{2};r^2\right).
\end{equation}
Moreover,
\begin{equation}\label{hyp:eq:Lambda-endpoints}
 \Lambda_n(0)=0,
 \qquad
 \lim_{r\uparrow1}\Lambda_n(r)=1,
 \qquad
 0<\Lambda_n(r)<1\quad(0<r<1),
\end{equation}
and
\begin{equation}\label{hyp:eq:Lambda-expansion}
 \Lambda_n(r)
 =\frac{2(n-1)}{n}r
 \left(1-\frac{n-2}{n+2}r^2+O(r^4)\right)
 \qquad(r\downarrow0).
\end{equation}
\end{proposition}

\begin{proof}
Consider the vector-valued Poisson extension of the identity boundary map,
\[
 V(x)=P_h[\id](x)=\int_{\Sph^{n-1}}\Ph(x,\eta)\eta\,d\sigma(\eta).
\]
Rotational covariance gives
\begin{equation}\label{hyp:eq:V-radial}
 V(r\xi)=R(r)\xi
\end{equation}
for a scalar function $R$, and by \eqref{eq:hyp-Lambda-intro}, $R=\Lambda_n$. The boundary data are continuous, so $R(r)\to1$ as $r\uparrow1$.

Let $Y(\xi)$ be a spherical harmonic of degree one. Substituting $R(r)Y(\xi)$ into $\Dh(RY)=0$ gives
\begin{equation}\label{hyp:eq:radial-ode}
 (1-r^2)\left(R''+\frac{n-1}{r}R'-\frac{n-1}{r^2}R\right)
 +2(n-2)rR'=0.
\end{equation}
Write $R(r)=r y(r^2)$ and set $s=r^2$. Then \eqref{hyp:eq:radial-ode} becomes
\begin{equation}\label{hyp:eq:hypergeom-ode}
 s(1-s)y''+
 \left(\frac{n+2}{2}+\frac{n-6}{2}s\right)y'
 +\frac{n-2}{2}y=0.
\end{equation}
This is Gauss' hypergeometric equation with
\[
 a=1,\qquad b=1-\frac n2,\qquad c=\frac{n+2}{2}.
\]
The solution regular at $s=0$ is therefore
\[
 y(s)=C\,{}_2F_1\!\left(1,1-\frac n2;\frac{n+2}{2};s\right).
\]
Gauss' value at $1$ gives
\begin{align*}
 {}_2F_1\!\left(1,1-\frac n2;\frac{n+2}{2};1\right)
 &=\frac{\Gamma((n+2)/2)\Gamma(n-1)}
 {\Gamma(n/2)\Gamma(n)}\\
 &=\frac{n}{2(n-1)}.
\end{align*}
The boundary normalization $\lim_{r\uparrow1}R(r)=1$ therefore forces $C=2(n-1)/n$, proving \eqref{hyp:eq:Lambda-explicit}.

The endpoint statements follow from the Poisson representation. Positivity can also be seen directly from
\[
 \Lambda_n(r)
 =c_n\int_0^{\pi/2}
 \bigl(p_r^h(t)-p_r^h(\pi-t)\bigr)
 \cos t\,\sin^{n-2}t\,dt>0,
\]
since $p_r^h$ is strictly decreasing. Strict inequality $\Lambda_n(r)<1$ follows because $V(r\xi)$ is a nontrivial positive average of points of $\Sph^{n-1}$. Finally, the Taylor series for ${}_2F_1$ gives \eqref{hyp:eq:Lambda-expansion}.
\end{proof}

\begin{remark}[Special dimensions]\label{hyp:rem:special-dimensions}
For $n=2$, \eqref{hyp:eq:Lambda-explicit} gives $\Lambda_2(r)=r$, as expected. For even dimensions the hypergeometric series terminates. For example,
\begin{align}\label{hyp:eq:special-lambda}
 \Lambda_4(r)&=\frac{3r-r^3}{2},\\
 \Lambda_6(r)&=\frac53 r\left(1-\frac12r^2+\frac1{10}r^4\right).\notag
\end{align}
Thus in dimension four the sharp volume profile is
\[
 |u(r\B^4)|\le\omega_4\left(\frac{3r-r^3}{2}\right)^4.
\]
\end{remark}

\subsection{Sharp hyperbolic mean-width, intrinsic-volume, and volume distortion}
\begin{proof}[Proof of Theorem~\ref{thm:hyp-main-intro}]
Because $u(\mathbb B^n)\subset\mathbb B^n$, the map is bounded.  By the boundary representation above there exists $F:\mathbb S^{n-1}\to\overline{\mathbb B^n}$ such that $u=P_h[F]$.  For $\xi\in\mathbb S^{n-1}$,
\[
 u(r\xi)=\int_{\mathbb S^{n-1}}P_h(r\xi,\eta)F(\eta)\,d\sigma(\eta).
\]
Since the kernel is strictly decreasing and has first multiplier $\Lambda_n(r)$, Theorem~\ref{thm:zonal-main} gives
\[
 w\!\left(\operatorname{co}u(r\mathbb S^{n-1})\right)\le2\Lambda_n(r).
\]
Let $K_r=\operatorname{co}u(r\mathbb S^{n-1})$.  By the hyperbolic-harmonic convex-hull property, $u(r\mathbb B^n)\subset K_r$.  Hence Urysohn's inequality yields
\[
 |u(r\mathbb B^n)|\le |K_r|
 \le\omega_n\left(\frac{w(K_r)}2\right)^n
 \le\omega_n\Lambda_n(r)^n.
\]
For $Q\in O(n)$, the map $U_Q=P_h[Q\,\mathrm{id}]$ satisfies
\[
 U_Q(r\xi)=\Lambda_n(r)Q\xi,
\]
so equality holds in the mean-width estimate.  Since $\Lambda_n$ is continuous, $\Lambda_n(0)=0$, and $\Lambda_n(r)>0$, the intermediate value theorem shows that $U_Q(r\mathbb B^n)$ contains $\Lambda_n(r)\mathbb B^n$; the upper bound therefore gives equality in volume as well.

Finally, equality in image volume forces equality throughout the chain
of inequalities above, and therefore equality in mean width.
Theorem~\ref{thm:zonal-rigidity} then gives $F(\eta)=Q\eta$ almost
everywhere.  Hence $u=U_Q$ throughout $\mathbb B^n$ in either equality case.
\end{proof}

\begin{corollary}[Hyperbolic intrinsic-volume hierarchy]\label{hyp:cor:hyp-intrinsic}
Let $K_r=\operatorname{co}u(r\mathbb S^{n-1})$. Then, for $1\le j\le n$,
\[
 V_j(K_r)\le \Lambda_n(r)^j V_j(\overline{\mathbb B^n}).
\]
Each inequality is sharp.  Equality for one $j$ at one radius forces $u=U_Q$ for some $Q\in O(n)$.
\end{corollary}
\begin{proof}
Combine the extended isoperimetric inequality \eqref{eq:extended-Urysohn} with Theorem~\ref{thm:hyp-main-intro}.  Equality in any intrinsic-volume estimate forces equality in mean width and hence the rigidity conclusion.
\end{proof}

\begin{corollary}[Failure of concentric Euclidean volume contraction]\label{hyp:cor:hyp-no-contraction}
If $n\ge3$, there exist hyperbolic-harmonic self-maps $u:\mathbb B^n\to\mathbb B^n$ and arbitrarily small $r>0$ such that
\[
 |u(r\mathbb B^n)|>|r\mathbb B^n|=\omega_nr^n.
\]
\end{corollary}
\begin{proof}
For $U_Q$, equality gives $|U_Q(r\mathbb B^n)|=\omega_n\Lambda_n(r)^n$, while the expansion of $\Lambda_n$ at the origin gives $\Lambda_n(r)/r\to2(n-1)/n>1$.
\end{proof}

For $a\in\B^n$, choose a hyperbolic automorphism $\varphi_a$ with $\varphi_a(0)=a$ and define the pseudohyperbolic ball
\begin{equation}\label{hyp:eq:pseudohyp-ball}
 B_h(a,r):=\varphi_a(r\B^n),\qquad 0<r<1.
\end{equation}
The set does not depend on the particular choice of $\varphi_a$ because two such choices differ by an orthogonal map fixing the origin.

\begin{corollary}[Sharp bounds on pseudohyperbolic balls]\label{hyp:cor:pseudohyp}
Let $u:\B^n\to\B^n$ be hyperbolic harmonic. Then for every $a\in\B^n$ and $0<r<1$,
\begin{align}
 w\!\left(\co u(\partial B_h(a,r))\right)&\le2\Lambda_n(r),\label{hyp:eq:pseudo-width}\\
 |u(B_h(a,r))|&\le\omega_n\Lambda_n(r)^n.\label{hyp:eq:pseudo-volume}
\end{align}
Both estimates are sharp.
\end{corollary}

\begin{proof}
The mapping $v=u\circ\varphi_a$ is hyperbolic harmonic and maps $\B^n$ into itself. Apply \cref{thm:hyp-main-intro} to $v$. Sharpness is obtained from $u=U_Q\circ\varphi_a^{-1}$.
\end{proof}

The limit as $r\downarrow0$ gives a bound on the derivative ellipsoid,
just as in the Euclidean consequence of Theorem~\ref{thm:main}.

\begin{proposition}[Mean width of the derivative ellipsoid at the origin]\label{hyp:prop:origin-derivative}
For every hyperbolic-harmonic $u:\B^n\to\B^n$,
\begin{equation}\label{hyp:eq:origin-width}
 w\!\left(Du(0)\overline{\B^n}\right)
 \le \frac{4(n-1)}n.
\end{equation}
Equivalently,
\begin{equation}\label{hyp:eq:origin-average}
 \int_{\Sph^{n-1}}|Du(0)^T\xi|\,d\sigma(\xi)
 \le\frac{2(n-1)}n.
\end{equation}
\end{proposition}

\begin{proof}
Let
\[
 K_r=\co u(r\Sph^{n-1}),
 \qquad
 L_r=\frac{K_r-u(0)}r.
\]
Differentiability at the origin gives, uniformly for $\xi\in\Sph^{n-1}$,
\[
 \frac{u(r\xi)-u(0)}r\longrightarrow Du(0)\xi.
\]
Hence $L_r$ converges in the Hausdorff metric to
\[
 \co\bigl(Du(0)\Sph^{n-1}\bigr)=Du(0)\overline{\B^n}.
\]
Mean width is continuous under Hausdorff convergence and translation invariant, so \cref{thm:hyp-main-intro} gives
\[
 w\!\left(Du(0)\overline{\B^n}\right)
 =\lim_{r\downarrow0}\frac{w(K_r)}r
 \le 2\lim_{r\downarrow0}\frac{\Lambda_n(r)}r
 =\frac{4(n-1)}n.
\]
Equation \eqref{hyp:eq:origin-average} follows from \eqref{eq:hyp-ellipsoid-width}.
\end{proof}

\begin{proof}[Proof of Theorem~\ref{thm:hyp-differential-intro}]
Choose a hyperbolic automorphism $\varphi$ with $\varphi(0)=x$ and put $v=u\circ\varphi$. By \eqref{hyp:eq:automorphism-derivative},
\[
 Dv(0)=Du(x)D\varphi(0)
 =(1-|x|^2)Du(x)O
\]
for some $O\in O(n)$. Since $O\overline{\B^n}=\overline{\B^n}$, \cref{hyp:prop:origin-derivative} yields
\[
 (1-|x|^2)w\!\left(Du(x)\overline{\B^n}\right)
 \le \frac{4(n-1)}n.
\]
Using \eqref{eq:hyp-ellipsoid-width}, we obtain
\begin{equation}\label{hyp:eq:average-derivative}
 \int_{\Sph^{n-1}}|Du(x)^T\xi|\,d\sigma(\xi)
 \le\frac{2(n-1)}{n(1-|x|^2)}.
\end{equation}

If $Du(x)$ is singular, \eqref{hyp:eq:jacobian} is immediate. Otherwise Urysohn's inequality applied to the derivative ellipsoid gives
\begin{align*}
 |\det Du(x)|^{1/n}
 &=\left(\frac{|Du(x)\B^n|}{\omega_n}\right)^{1/n}\\
 &\le \frac12w\!\left(Du(x)\overline{\B^n}\right)\\
 &\le\frac{2(n-1)}{n(1-|x|^2)}.
\end{align*}
Consequently,
\begin{equation}\label{hyp:eq:jacobian}
 |\det Du(x)|\le
 \left(\frac{2(n-1)}{n(1-|x|^2)}\right)^n.
\end{equation}

For sharpness at a prescribed $x$, choose $\varphi$ with $\varphi(0)=x$ and take
\[
 u=U_Q\circ\varphi^{-1}.
\]
Since $DU_Q(0)=\frac{2(n-1)}nQ$, the derivative $Du(x)$ is a scalar multiple of an orthogonal map with scalar
\[
 \frac{2(n-1)}{n(1-|x|^2)}.
\]
Thus equality holds in both estimates.
\end{proof}

\begin{remark}[Relation to pointwise Schwarz theory]
The estimates above control the mean width of the derivative ellipsoid
and the determinant. They are therefore different from sharp operator-norm bounds for $Du$ and from sharp pointwise estimates for $|u(x)|$. Hyperbolic-harmonic Schwarz lemmas of those types were developed by Burgeth \cite{Burgeth1992}, Chen--Kalaj \cite{ChenKalaj2021}, and Khalfallah--Haggui--Mateljevi\'c \cite{KhalfallahHagguiMateljevic2022}. For sharp Schwarz--Pick estimates for Euclidean harmonic maps that are conformal at a prescribed point, and for the corresponding extremal theory of conformal minimal discs, see Forstneri\v{c}--Kalaj \cite{ForstnericKalaj2024}. In particular, the present determinant constant should not be confused with the sharp constant for the operator norm of $Du(0)$.
\end{remark}
\begin{remark}[$\alpha$-harmonic interpolation]
The Euclidean and hyperbolic Poisson kernels considered above may be viewed
as two distinguished members of a common $\alpha$-harmonic family.  Namely,
for suitable $\alpha$ one considers
\[
 P_{\alpha}(x,\zeta)
 =
 C_{n,\alpha}
 \frac{(1-|x|^2)^{1-\alpha}}
 {|x-\zeta|^{\,n-\alpha}},
 \qquad
 x\in\mathbb B^n,\quad
 \zeta\in\mathbb S^{n-1},
\]
where
\[
 C_{n,\alpha}
 =
 \frac{
 \Gamma\!\left(\frac{n-\alpha}{2}\right)
 \Gamma\!\left(1-\frac{\alpha}{2}\right)}
 {\Gamma\!\left(\frac n2\right)\Gamma(1-\alpha)}.
\]
For $x=r\xi$ this is a zonal kernel,
\[
 p_{\alpha,r}(t)
 =
 C_{n,\alpha}
 \frac{(1-r^2)^{1-\alpha}}
 {(1-2r\cos t+r^2)^{(n-\alpha)/2}}.
\]
The choice $\alpha=0$ gives the ordinary Poisson kernel,
whereas $\alpha=2-n$ gives the Poisson--Szeg\H{o} kernel
\[
 \left(
 \frac{1-r^2}{1-2r\cos t+r^2}
 \right)^{n-1}.
\]
Thus the Euclidean and hyperbolic-harmonic results of this paper may be
regarded as two specializations of a single zonal family.  Whenever
$p_{\alpha,r}$ is decreasing in $t$, Theorem~\ref{thm:zonal-main} applies
directly, with sharp contraction factor given by the degree-one multiplier
$\lambda_1(p_{\alpha,r})$. A direct evaluation of \eqref{eq:intro-lambda1} gives \[
\lambda_{1,\alpha}(r)
=
\frac{
\Gamma\!\left(\frac{n-\alpha}{2}\right)
\Gamma\!\left(1-\frac{\alpha}{2}\right)}
{\Gamma\!\left(\frac n2\right)\Gamma(1-\alpha)}
\frac{n-\alpha}{n}\,
r\,
{}_2F_1\!\left(
\frac{\alpha}{2},
\frac{n+\alpha}{2};
\frac{n+2}{2};
r^2
\right).
\]
\end{remark}
\section{Euclidean harmonic estimates: the exact pointwise Jacobian profile}
\label{local:sec:euclidean-jacobian}

We now determine the pointwise determinant supremum for ordinary harmonic
self-maps.  For $n\ge3$, harmonicity is not preserved by precomposition
with M\"obius automorphisms of the ball, so the centered estimate alone
does not give a sharp bound elsewhere.  Instead, the Poisson gradient
turns the problem into a finite-dimensional minimization over positive
definite matrices.  This is a determinant variational principle of
Lewis type \cite{Lewis1979}.

We first prove the matrix formula and construct an extremal.  We then
reduce it by symmetry to a unique scalar minimum and establish the
anisotropy of the optimal derivative.  Finally, we derive and evaluate
a quadratic relaxation in terms of the Poisson-information matrix.

Let
\begin{equation}\label{local:eq:euclidean-Poisson-kernel}
 P(x,\eta)
 :=
 \frac{1-|x|^2}{|x-\eta|^n},
 \qquad
 x\in\B^n,\quad \eta\in\Sph^{n-1},
\end{equation}
be the Euclidean Poisson kernel, with $\sigma$ normalized so that
\[
 \int_{\Sph^{n-1}}P(x,\eta)\,d\sigma(\eta)=1.
\]
For fixed $x\in\B^n$, put
\begin{equation}\label{local:eq:gx-definition}
 g_x(\eta):=\nabla_xP(x,\eta)
\end{equation}
and, for a positive definite symmetric matrix $C$, define
\begin{equation}\label{local:eq:LxC-definition}
 \mathcal L_x(C)
 :=
 \int_{\Sph^{n-1}}|C g_x(\eta)|\,d\sigma(\eta).
\end{equation}

\subsection{An exact determinant variational principle}

\begin{theorem}[Exact Euclidean Jacobian extremal principle]
\label{local:thm:exact-euclidean-jacobian}
For every $x\in\B^n$,
\begin{equation}\label{local:eq:exact-matrix-profile}
 \sup\left\{
 |\det Df(x)|:
 f:\B^n\to\B^n\ \text{harmonic}
 \right\}
 =
 \frac1{n^n}
 \inf_{C>0}
 \frac{\mathcal L_x(C)^n}{\det C},
\end{equation}
where the infimum is over positive definite symmetric matrices.  A minimizing matrix exists; its positive scalar multiples give the same value.

If $C_x$ is a minimizing matrix, then the boundary map
\begin{equation}\label{local:eq:exact-extremal-boundary-general}
 F_x(\eta)
 :=
 \frac{C_xg_x(\eta)}{|C_xg_x(\eta)|}
\end{equation}
has values in $\Sph^{n-1}$, and its Poisson extension $f_x=P[F_x]$ is an
extremal.  Moreover,
\begin{equation}\label{local:eq:exact-extremal-derivative-general}
 Df_x(x)
 =
 \frac{\mathcal L_x(C_x)}{n}\,C_x^{-1}.
\end{equation}
\end{theorem}

\begin{proof}
Let $f:\B^n\to\B^n$ be harmonic.  Since $f$ is bounded, it has a Poisson
representation
\[
 f(x)
 =
 \int_{\Sph^{n-1}}P(x,\eta)F(\eta)\,d\sigma(\eta),
 \qquad |F(\eta)|\le1\quad\text{a.e.},
\]
and hence
\begin{equation}\label{local:eq:Df-Poisson-gradient}
 Df(x)
 =
 \int_{\Sph^{n-1}}F(\eta)g_x(\eta)^T\,d\sigma(\eta).
\end{equation}
Write the polar decomposition as \(Df(x)=QH\), where \(Q\in O(n)\) is the orthogonal factor, representing the rotational/reflection part of \(Df(x)\), and
\[
H=\bigl(Df(x)^T Df(x)\bigr)^{1/2}\ge 0
\]
is the symmetric positive semidefinite factor, representing the stretching part of \(Df(x)\).
After replacing $F$ by $Q^TF$, equation \eqref{local:eq:Df-Poisson-gradient}
becomes
\[
 H
 =
 \int_{\Sph^{n-1}}\widetilde F(\eta)g_x(\eta)^T\,d\sigma(\eta),
 \qquad |\widetilde F|\le1.
\]
Therefore, for every $C>0$,
\[
 \tr(CH)
 =
 \int_{\Sph^{n-1}}
 \langle \widetilde F(\eta),Cg_x(\eta)\rangle\,d\sigma(\eta)
 \le
 \mathcal L_x(C).
\]
Applying the arithmetic--geometric mean inequality to
$C^{1/2}HC^{1/2}$ gives
\[
 (\det C\,\det H)^{1/n}
 \le
 \frac1n\tr(CH),
\]
so
\begin{equation}\label{local:eq:matrix-upper-bound}
 |\det Df(x)|
 \le
 \frac{\mathcal L_x(C)^n}{n^n\det C}.
\end{equation}
Taking the infimum proves one inequality in \eqref{local:eq:exact-matrix-profile}.

The quotient in \eqref{local:eq:exact-matrix-profile} is invariant under positive
rescaling of $C$, so we may impose $\det C=1$.  Differentiating Poisson reproduction of the coordinate functions gives
\[
 \int_{\Sph^{n-1}}\eta g_x(\eta)^T\,d\sigma(\eta)=I.
\]
Thus $g_x(\eta)$ spans $\R^n$, and $C\mapsto\mathcal L_x(C)$ is a
norm on the space of symmetric matrices.  A bounded minimizing sequence
with $\det C=1$ therefore has uniformly bounded eigenvalues.  Their
product is $1$, so they are also bounded away from zero.  A subsequence
converges to a positive definite minimizer.

The denominator in the extremal data never vanishes.  Indeed,
$g_x=P\nabla_x\log P$, and the score identity
\eqref{local:eq:score-sphere-new} gives
\[
 |\nabla_x\log P(x,\eta)|
 \ge\frac{n-(n-2)|x|}{1-|x|^2}>0.
\]
Since the sphere is compact, differentiation under the following
integrals is justified.

Fix such a minimizer $C$ and set
\begin{equation}\label{local:eq:M-C-definition}
 M_C
 :=
 \int_{\Sph^{n-1}}
 \frac{g_x(\eta)g_x(\eta)^T}{|Cg_x(\eta)|}\,d\sigma(\eta).
\end{equation}
The first variation of $\mathcal L_x$ under the constraint $\det C=1$ gives
\begin{equation}\label{local:eq:Euler-matrix}
 \frac12(CM_C+M_CC)
 =
 \frac{\mathcal L_x(C)}{n}\,C^{-1}.
\end{equation}
Indeed, the Lagrange multiplier is $\mathcal L_x(C)/n$, as follows by
multiplying \eqref{local:eq:Euler-matrix} by $C$ and taking the trace.  Diagonalize
$C$.  The off-diagonal entries of \eqref{local:eq:Euler-matrix} then show that
$M_C$ is diagonal in the same basis, while the diagonal entries give
\begin{equation}\label{local:eq:Lewis-isotropy}
 M_C
 =
 \frac{\mathcal L_x(C)}{n}\,C^{-2}.
\end{equation}

Now use the boundary data
\[
 F_C(\eta)
 =
 \frac{Cg_x(\eta)}{|Cg_x(\eta)|}.
\]
Its Poisson extension takes values in the closed unit ball.  By \eqref{local:eq:M-C-definition} and
\eqref{local:eq:Lewis-isotropy},
\[
 Df_C(x)
 =
 \int_{\Sph^{n-1}}
 \frac{Cg_x(\eta)g_x(\eta)^T}{|Cg_x(\eta)|}\,d\sigma(\eta)
 =
 CM_C
 =
 \frac{\mathcal L_x(C)}nC^{-1}.
\]
The derivative is invertible, so $f_C$ is nonconstant.  The strong maximum
principle applied to the subharmonic function $|f_C|^2$ now gives
$f_C(\B^n)\subset\B^n$.  Moreover,
\[
 |\det Df_C(x)|
 =
 \frac{\mathcal L_x(C)^n}{n^n\det C},
\]
which proves equality and the extremal assertion.
\end{proof}

\begin{remark}
Equation \eqref{local:eq:Lewis-isotropy} is an isotropic ellipsoid condition of the
same general type that appears in determinant formulations of John's theorem.
The abstract ellipsoidal viewpoint goes back, in particular, to Lewis
\cite{Lewis1979}; the point here is that the Poisson-gradient family
$g_x(\eta)$ makes this principle an exact Jacobian extremal formula for
vector-valued harmonic self-maps.
\end{remark}

\subsection{The sharp radial profile and its extremals}

By rotational invariance, the left-hand side of
\eqref{local:eq:exact-matrix-profile} depends only on $r=|x|$.  Define
\begin{equation}\label{local:eq:Jn-definition}
 \mathcal J_n(r)
 :=
 \sup\left\{
 |\det Df(x)|:
 f:\B^n\to\B^n\ \text{harmonic},\quad |x|=r
 \right\}.
\end{equation}
For $x=re_n$, write
\[
 g_r(\eta)=g_{re_n}(\eta)=(g_r'(\eta),g_{r,n}(\eta))
 \in\R^{n-1}\times\R,
\]
and define, for $q>0$,
\begin{equation}\label{local:eq:Lnrq-definition}
 L_{n,r}(q)
 :=
 \int_{\Sph^{n-1}}
 \sqrt{|g_r'(\eta)|^2+q^2g_{r,n}(\eta)^2}\,d\sigma(\eta).
\end{equation}

\begin{theorem}[Sharp radial Euclidean Jacobian profile]
\label{local:thm:sharp-radial-euclidean-profile}
For every $n\ge2$ and $0\le r<1$,
\begin{equation}\label{local:eq:Jn-exact-one-parameter}
 \mathcal J_n(r)
 =
 \frac1{n^n}
 \min_{q>0}
 \frac{L_{n,r}(q)^n}{q}.
\end{equation}
The minimizer $q=q_n(r)$ is unique and is characterized by
\begin{equation}\label{local:eq:q-stationarity}
 n q_n(r)L_{n,r}'(q_n(r))
 =
 L_{n,r}(q_n(r)).
\end{equation}

Put
\begin{equation}\label{local:eq:D-N-definition}
 D_r(t):=1+r^2-2rt
\end{equation}
and
\begin{equation}\label{local:eq:Nnr-definition}
 N_{n,r}(t)
 :=
 nt-(n+2)r+(4-n)r^2t+(n-2)r^3.
\end{equation}
Then
\begin{align}
 L_{n,r}(q)
 &=c_n\int_{-1}^1
 \frac{
 \sqrt{
 n^2(1-r^2)^2(1-t^2)+q^2N_{n,r}(t)^2
 }
 }{
 D_r(t)^{n/2+1}
 }
 (1-t^2)^{(n-3)/2}\,dt,
 \label{local:eq:Lnrq-one-dimensional}\\
 c_n
 &=
 \frac{\Gamma(n/2)}{\sqrt\pi\,\Gamma((n-1)/2)}.
 \notag
\end{align}

An extremal boundary map at $re_n$ is, up to a target orthogonal map $Q$,
\begin{equation}\label{local:eq:radial-extremal-boundary}
 F_{n,r}(\eta)
 =
 Q\,
 \frac{
 \bigl(
 n(1-r^2)\eta',\,
 q_n(r)N_{n,r}(\eta_n)
 \bigr)
 }{
 \sqrt{
 n^2(1-r^2)^2(1-\eta_n^2)
 +q_n(r)^2N_{n,r}(\eta_n)^2
 }
 }.
\end{equation}
For its Poisson extension,
\begin{equation}\label{local:eq:radial-extremal-derivative}
 Df(re_n)
 =
 Q\,
 \frac{L_{n,r}(q_n(r))}{n}
 \operatorname{diag}\left(1,\ldots,1,\frac1{q_n(r)}\right).
\end{equation}
\end{theorem}

\begin{proof}
Let $K=O(n-1)$ act on $\R^n$ by rotations fixing $e_n$.  For $R\in K$,
rotational covariance of the Poisson kernel gives
\[
 g_r(R\eta)=Rg_r(\eta)
\]
and hence
\[
 \mathcal L_{re_n}(RCR^T)=\mathcal L_{re_n}(C).
\]
Average a positive definite $C$ over $K$:
\[
 \overline C=\int_K RCR^T\,dR.
\]
Convexity of $C\mapsto\mathcal L_{re_n}(C)$ gives
\[
 \mathcal L_{re_n}(\overline C)
 \le
 \mathcal L_{re_n}(C),
\]
whereas concavity of $\log\det$ gives
\[
 \det\overline C\ge\det C.
\]
Thus the quotient in \eqref{local:eq:exact-matrix-profile} does not increase under
this averaging.  Every $K$-invariant positive definite matrix has the form
\[
 \operatorname{diag}(a,\ldots,a,b),
 \qquad a,b>0.
\]
Using scale invariance and setting $q=b/a$ reduces
\eqref{local:eq:exact-matrix-profile} exactly to
\eqref{local:eq:Jn-exact-one-parameter}.

To prove uniqueness, write $q=e^s$.  For almost every $\eta$, both
$|g_r'(\eta)|$ and $|g_{r,n}(\eta)|$ are nonzero, and
\[
 s\longmapsto
 \sqrt{|g_r'(\eta)|^2+e^{2s}g_{r,n}(\eta)^2}
\]
is strictly log-convex.  H\"older's inequality therefore implies that
$s\mapsto\log L_{n,r}(e^s)$ is strictly convex.  Its derivative tends to $0$
as $s\to-\infty$ and to $1$ as $s\to+\infty$.  Hence
\[
 s\longmapsto n\log L_{n,r}(e^s)-s
\]
has a unique critical point, which is its unique minimum.  Differentiation
gives \eqref{local:eq:q-stationarity}.

It remains to compute $g_r$.  From \eqref{local:eq:euclidean-Poisson-kernel},
\[
 \nabla_xP(x,\eta)
 =
 -2x|x-\eta|^{-n}
 -n(1-|x|^2)(x-\eta)|x-\eta|^{-n-2}.
\]
At $x=re_n$, with $t=\eta_n$, this gives
\begin{align*}
 g_r'(\eta)
 &=
 \frac{n(1-r^2)\eta'}{D_r(t)^{n/2+1}},\\
 g_{r,n}(\eta)
 &=
 \frac{N_{n,r}(t)}{D_r(t)^{n/2+1}}.
\end{align*}
The standard slicing formula on the sphere yields
\eqref{local:eq:Lnrq-one-dimensional}.  Formula
\eqref{local:eq:radial-extremal-boundary} is exactly
\eqref{local:eq:exact-extremal-boundary-general} for
$C=\operatorname{diag}(1,\ldots,1,q_n(r))$, and
\eqref{local:eq:radial-extremal-derivative} follows from
\eqref{local:eq:exact-extremal-derivative-general}.
\end{proof}

\begin{proposition}[Strict anisotropy of the Euclidean extremal]
\label{local:prop:q-less-than-one}
Let $n\ge3$, $0<r<1$, and let $q_n(r)$ be the unique minimizer in
\eqref{local:eq:Jn-exact-one-parameter}.
Then
\[
0<q_n(r)<1.
\]
In particular, for \(n\ge3\) and \(r>0\), the sharp extremal derivative is
anisotropic: its radial singular value is strictly larger than its
tangential singular values.
\end{proposition}

\begin{proof}
Write $L(q)=L_{n,r}(q)$ and set $\Psi(s)=n\log L(e^s)-s$.
The strict convexity proved in
Theorem~\ref{local:thm:sharp-radial-euclidean-profile} shows that its
unique minimizer is $\log q_n(r)$.  It therefore suffices to prove
$\Psi'(0)>0$.

Differentiating at \(q=1\), we obtain
\begin{equation}\label{local:eq:Psi-prime-one}
\Psi'(0)
=
\frac{nL'(1)-L(1)}{L(1)},
\end{equation}
where
\begin{equation}\label{local:eq:q-one-numerator}
nL'(1)-L(1)
=
\int_{\Sph^{n-1}}
\frac{
n g_{r,n}(\eta)^2-|g_r(\eta)|^2
}{
|g_r(\eta)|
}
\,d\sigma(\eta).
\end{equation}

We prove that the integral in \eqref{local:eq:q-one-numerator} is strictly
positive. Set
\[
x=re_n,
\qquad
A:=1-r^2,
\]
and introduce \(\omega\in\Sph^{n-1}\) by
\begin{equation}\label{local:eq:omega-change}
\frac{\eta-x}{|\eta-x|^2}
=
\frac{x+\omega}{A}.
\end{equation}
Equivalently,
\[
\eta
=
x+A\frac{x+\omega}{|x+\omega|^2}.
\]
The corresponding change of variables satisfies
\begin{equation}\label{local:eq:harmonic-measure-change}
P(x,\eta)\,d\sigma(\eta)
=
|x+\omega|^{2-n}\,d\sigma(\omega).
\end{equation}

Let
\[
S_x(\eta):=\nabla_x\log P(x,\eta).
\]
Since
\[
g_r(\eta)=P(x,\eta)S_x(\eta),
\]
and
\[
S_x(\eta)
=
-\frac{2x}{A}
+n\frac{\eta-x}{|\eta-x|^2},
\]
the change of variables \eqref{local:eq:omega-change} gives
\begin{equation}\label{local:eq:score-omega}
S_x(\eta)
=
\frac{n\omega+(n-2)x}{A}.
\end{equation}

Write
\[
t:=\omega_n,
\qquad
a:=(n-2)r,
\]
and set
\[
H(t):=1+r^2+2rt,
\qquad
Z(t):=n^2+a^2+2nat.
\]
Using \eqref{local:eq:harmonic-measure-change} and
\eqref{local:eq:score-omega}, the sign of
\eqref{local:eq:q-one-numerator} is the sign of
\begin{equation}\label{local:eq:Jphi-integral}
\int_{-1}^1 J(t)\phi(t)\,d\nu_n(t),
\end{equation}
where
\[
\phi(t)
:=
H(t)^{-(n-2)/2}Z(t)^{-1/2},
\]
\[
J(t)
:=
n^2(nt^2-1)
+2na(n-1)t
+(n-1)a^2,
\]
and
\[
d\nu_n(t)
=
c_n(1-t^2)^{(n-3)/2}\,dt.
\]

For smooth \(\phi\), spherical integration by parts gives
\begin{equation}\label{local:eq:spherical-ibp-first}
\int_{-1}^1 t\phi(t)\,d\nu_n(t)
=
\frac{1}{n-1}
\int_{-1}^1
(1-t^2)\phi'(t)\,d\nu_n(t)
\end{equation}
and
\begin{equation}\label{local:eq:spherical-ibp-second}
\int_{-1}^1 (nt^2-1)\phi(t)\,d\nu_n(t)
=
\frac{1}{n+1}
\int_{-1}^1
(1-t^2)^2\phi''(t)\,d\nu_n(t).
\end{equation}
Consequently, \eqref{local:eq:Jphi-integral} equals
\begin{equation}\label{local:eq:Jphi-after-ibp}
\int_{-1}^1
\left[
\frac{n^2}{n+1}(1-t^2)^2\phi''
+2na(1-t^2)\phi'
+(n-1)a^2\phi
\right]
d\nu_n.
\end{equation}

A direct differentiation gives
\[
\frac{\phi'}{\phi}
=
-a\left(\frac1H+\frac nZ\right)
\]
and
\[
\frac{\phi''}{\phi}
=
a^2\left(\frac1H+\frac nZ\right)^2
+
\frac{2a^2}{(n-2)H^2}
+
\frac{2n^2a^2}{Z^2}.
\]
Define
\[
X(t):=\frac{1-t^2}{H(t)},
\qquad
Y(t):=\frac{n(1-t^2)}{Z(t)}.
\]
The integrand in \eqref{local:eq:Jphi-after-ibp} is then
\[
a^2\phi(t)\,B_n(X(t),Y(t)),
\]
where
\[
B_n(X,Y)
=
\frac{n^2}{n+1}
\left(
\frac{n}{n-2}X^2+2XY+3Y^2
\right)
-2n(X+Y)+(n-1).
\]
Completing the square yields
\begin{align}
B_n(X,Y)
=
\frac{n^2}{n+1}
\Bigg[
&
\frac{n}{n-2}
\left(
X-\frac{n-2}{n}
\right)^2
\nonumber\\
&+
2
\left(
X-\frac{n-2}{n}
\right)
\left(
Y-\frac1n
\right)
+
3
\left(
Y-\frac1n
\right)^2
\Bigg].
\label{local:eq:Bn-positive}
\end{align}
The quadratic form in \eqref{local:eq:Bn-positive} is positive definite, since
\[
\det
\begin{pmatrix}
\dfrac{n}{n-2} & 1\\[1mm]
1 & 3
\end{pmatrix}
=
\frac{2(n+1)}{n-2}>0.
\]
Hence
\[
B_n(X,Y)\ge0.
\]
Moreover, \(B_n(X(t),Y(t))\) is not identically zero. Indeed,
\[
X(t),Y(t)\longrightarrow0
\qquad\text{as }t\to\pm1,
\]
and therefore
\[
B_n(X(t),Y(t))\longrightarrow n-1>0.
\]
Since \(a=(n-2)r>0\), \(\phi>0\), and \(d\nu_n\) is positive on
\((-1,1)\), it follows that
\[
nL'(1)-L(1)>0.
\]
Equation~\eqref{local:eq:Psi-prime-one} gives $\Psi'(0)>0$.
Strict convexity then forces $\log q_n(r)<0$, as required.
\end{proof}

\subsection{The Poisson-information bound as a quadratic relaxation}

Define the information matrix
\begin{equation}\label{local:eq:information-matrix-new}
 \mathcal I_n(x)
 :=
 \int_{\Sph^{n-1}}
 \frac{
 \nabla_xP(x,\eta)\nabla_xP(x,\eta)^T
 }{
 P(x,\eta)
 }
 \,d\sigma(\eta).
\end{equation}
Equivalently,
\[
 \mathcal I_n(x)
 =
 \int_{\Sph^{n-1}}
 P(x,\eta)
 \nabla_x\log P(x,\eta)
 \nabla_x\log P(x,\eta)^T
 \,d\sigma(\eta).
\]

\begin{theorem}[Poisson-information determinant bound]
\label{local:thm:euclidean-poisson-information}
Let $f:\B^n\to\B^n$ be harmonic.  Then, for every $x\in\B^n$,
\begin{equation}\label{local:eq:information-bound-new}
 |\det Df(x)|
 \le
 n^{-n/2}\sqrt{\det\mathcal I_n(x)}.
\end{equation}
The estimate is the quadratic $L^2$ relaxation of the exact variational
formula \eqref{local:eq:exact-matrix-profile}.
\end{theorem}

\begin{proof}
For every $C>0$, Cauchy--Schwarz with the weight $P(x,\eta)$ gives
\begin{align*}
 \mathcal L_x(C)^2
 &=
 \left(
 \int_{\Sph^{n-1}}
 \sqrt{P(x,\eta)}
 \frac{|Cg_x(\eta)|}{\sqrt{P(x,\eta)}}
 \,d\sigma(\eta)
 \right)^2\\
 &\le
 \tr\!\left(C^2\mathcal I_n(x)\right).
\end{align*}
Therefore Theorem~\ref{local:thm:exact-euclidean-jacobian} yields
\[
 |\det Df(x)|
 \le
 \frac1{n^n}
 \inf_{C>0}
 \frac{\tr(C^2\mathcal I_n(x))^{n/2}}{\det C}.
\]
Applying arithmetic--geometric mean to
$\mathcal I_n(x)^{1/2}C^2\mathcal I_n(x)^{1/2}$ gives
\[
 \tr(C^2\mathcal I_n(x))
 \ge
 n(\det C)^{2/n}(\det\mathcal I_n(x))^{1/n}.
\]
Equality is obtained when $C^2$ is a positive scalar multiple of
$\mathcal I_n(x)^{-1}$.  This proves \eqref{local:eq:information-bound-new}.
\end{proof}

\begin{proposition}[Euclidean Poisson-information matrix]
\label{local:prop:euclidean-poisson-information}
Let $x\in\B^n$ and put $r=|x|$.  Then
\begin{equation}\label{local:eq:I-explicit-new}
 \mathcal I_n(x)
 =
 \frac{n}{(1-r^2)^2}
 \left[
 \left(1-\frac{n-2}{n+2}r^2\right)I
 +
 \frac{4(n-2)}{n(n+2)}xx^T
 \right].
\end{equation}
Consequently, the tangential eigenvalue, of multiplicity $n-1$, is
\begin{equation}\label{local:eq:lambda-T-new}
 \lambda_T(r)
 =
 \frac{n}{(1-r^2)^2}
 \left(1-\frac{n-2}{n+2}r^2\right),
\end{equation}
whereas the radial eigenvalue is
\begin{equation}\label{local:eq:lambda-R-new}
 \lambda_R(r)
 =
 \frac{n}{(1-r^2)^2}
 \left(
 1-\frac{(n-2)(n-4)}{n(n+2)}r^2
 \right).
\end{equation}
\end{proposition}

\begin{proof}
Put $D(x,\eta)=|x-\eta|^2$ and let $\mathbb E_x$ denote expectation with
respect to $d\mu_x(\eta)=P(x,\eta)d\sigma(\eta)$.  If
$S=\nabla_x\log P(x,\eta)$, then $\mathbb E_xS=0$ and
\[
 S
 =
 -\frac{2x}{1-|x|^2}
 -n\frac{x-\eta}{D(x,\eta)}.
\] 
Introduce the symmetric positive-semidefinite matrix
\[
 C(x)
 :=
 \mathbb E_x
 \left[
 \frac{(x-\eta)(x-\eta)^T}{D(x,\eta)^2}
 \right],
 \qquad D(x,\eta)=|x-\eta|^2,
\]
where \(\mathbb E_x\) denotes expectation with respect to
\(d\mu_x(\eta)=P(x,\eta)\,d\sigma(\eta)\). Equivalently,
\[
 C(x)
 =
 (1-|x|^2)
 \int_{\Sph^{n-1}}
 \frac{(x-\eta)(x-\eta)^T}{|x-\eta|^{n+4}}
 \,d\sigma(\eta).
\]
A direct expansion gives
\begin{equation}\label{local:eq:I-via-C-new}
 \mathcal I_n(x)
 =
 n^2C(x)-\frac{4xx^T}{(1-|x|^2)^2}.
\end{equation}
On the other hand, the standard information identity
$\mathcal I_n(x)=-\mathbb E_x[D_x^2\log P]$ gives
\[
 \mathcal I_n(x)
 =
 \left(
 \frac{2}{1-|x|^2}+nR(x)
 \right)I
 +\frac{4xx^T}{(1-|x|^2)^2}
 -2nC(x),
\]
where
\[
 R(x)
 :=
 \mathbb E_x\left[\frac1{D(x,\eta)}\right].
\]
Eliminating $C(x)$,
\begin{equation}\label{local:eq:I-via-R-new}
 \mathcal I_n(x)
 =
 \frac{n}{n+2}
 \left(
 \frac{2}{1-|x|^2}+nR(x)
 \right)I
 +
 \frac{4(n-2)}{n+2}
 \frac{xx^T}{(1-|x|^2)^2}.
\end{equation}

The normalization of the Poisson kernel gives
\[
 \int_{\Sph^{n-1}}|x-\eta|^{-n}\,d\sigma(\eta)
 =
 \frac1{1-|x|^2}.
\]
Applying the Euclidean Laplacian and using
$\Delta_x|x-\eta|^{-n}=2n|x-\eta|^{-n-2}$ yields, for $r=|x|$,
\[
 \int_{\Sph^{n-1}}|x-\eta|^{-n-2}\,d\sigma(\eta)
 =
 \frac{n-(n-4)r^2}{n(1-r^2)^3}.
\]
Hence
\[
 R(x)
 =
 \frac{n-(n-4)r^2}{n(1-r^2)^2}.
\]
Substitution in \eqref{local:eq:I-via-R-new} gives
\eqref{local:eq:I-explicit-new}.  The eigenvalue formulas follow immediately.
\end{proof}

\begin{corollary}[Explicit information bound]
\label{local:cor:euclidean-jacobian}
Let $f:\B^n\to\B^n$ be harmonic and put $r=|x|$.  Then
\begin{equation}\label{local:eq:Theta-bound-new}
 |\det Df(x)|
 \le
 \frac{\Theta_n(r)}{(1-r^2)^n},
\end{equation}
where
\begin{equation}\label{local:eq:Theta-definition-new}
 \Theta_n(r)
 :=
 \left(1-\frac{n-2}{n+2}r^2\right)^{(n-1)/2}
 \left(
 1-\frac{(n-2)(n-4)}{n(n+2)}r^2
 \right)^{1/2}.
\end{equation}
In particular,
\begin{equation}\label{local:eq:simple-euclidean-bound-new}
 |\det Df(x)|
 \le
 \frac1{(1-|x|^2)^n}.
\end{equation}
At the origin the first estimate is sharp.
\end{corollary}

\begin{proof}
Since $\lambda_T(r)$ has multiplicity $n-1$,
\[
 \det\mathcal I_n(x)
 =
 \lambda_T(r)^{n-1}\lambda_R(r).
\]
Theorem~\ref{local:thm:euclidean-poisson-information} therefore gives
\eqref{local:eq:Theta-bound-new}.  For $n=2$, $\Theta_2(r)=1$.  For $n\ge4$,
both factors in \eqref{local:eq:Theta-definition-new} are at most $1$.  If $n=3$,
then, with $s=r^2$,
\[
 \Theta_3(r)^2-1
 =
 \frac{s(s^2+5s-125)}{375}<0,
 \qquad 0<s<1.
\]
Thus $\Theta_n(r)\le1$ in every dimension.
\end{proof}

\begin{proposition}[Strictness of the quadratic relaxation]
\label{local:prop:euclidean-information-rigidity}
If $n\ge3$ and $0<r<1$, then
\begin{equation}\label{local:eq:Jn-strict-information}
 \mathcal J_n(r)
 <
 \frac{\Theta_n(r)}{(1-r^2)^n}.
\end{equation}
The explicit information bound is therefore strictly larger than the
Jacobian supremum in this regime.
\end{proposition}

\begin{proof}
Fix $x\ne0$.  In the proof of Theorem~\ref{local:thm:euclidean-poisson-information},
the quadratic minimization is attained by
$C_0^2$ proportional to $\mathcal I_n(x)^{-1}$.  If the exact profile equaled
the information bound, then Cauchy--Schwarz would also be an equality for this
$C_0$.  Writing
\[
 S_x(\eta):=\nabla_x\log P(x,\eta),
\]
this would force $|C_0S_x(\eta)|$ to be constant on the sphere.

The score locus can be computed explicitly.  From
\[
 S_x(\eta)
 =
 -\frac{2x}{1-|x|^2}
 +n\frac{\eta-x}{|\eta-x|^2}
\]
and inversion of $\Sph^{n-1}$ one obtains
\begin{equation}\label{local:eq:score-sphere-new}
 \left|
 S_x(\eta)-\frac{n-2}{1-|x|^2}x
 \right|
 =
 \frac{n}{1-|x|^2}.
\end{equation}
Thus the score sphere has center
\[
 c=\frac{n-2}{1-|x|^2}x.
\]
If $M=C_0^2>0$ and $|C_0S_x|$ were constant, then for some $R>0$,
\[
 (c+R\omega)^TM(c+R\omega)
\]
would be independent of $\omega\in\Sph^{n-1}$.  Comparing $\omega$ and
$-\omega$ gives $c^TM\omega=0$ for every $\omega$, hence $Mc=0$.  Since
$M>0$, this forces $c=0$, i.e. $(n-2)x=0$, a contradiction.
\end{proof}

\subsection{The sharp planar case}

In dimension two the variational profile collapses to the familiar conformally
invariant formula.

\begin{theorem}[Sharp planar Jacobian estimate]
\label{local:thm:planar-jacobian}
Let $f:\mathbb D\to\mathbb D$ be harmonic.  Then
\begin{equation}\label{local:eq:planar-jacobian-new}
 |\det Df(z)|
 \le
 \frac1{(1-|z|^2)^2},
 \qquad z\in\mathbb D.
\end{equation}
The estimate is sharp at every prescribed point.  Equivalently,
\begin{equation}\label{local:eq:J2-exact}
 \mathcal J_2(r)=\frac1{(1-r^2)^2}.
\end{equation}
In the one-parameter formula \eqref{local:eq:Jn-exact-one-parameter}, the unique
minimizer is $q_2(r)=1$.
\end{theorem}

\begin{proof}
The estimate follows from Corollary~\ref{local:cor:euclidean-jacobian}, since
$\Theta_2(r)=1$.  For sharpness, fix $a\in\mathbb D$ and take a conformal
automorphism $\psi$ of $\mathbb D$ with $\psi(a)=0$.  Then
\[
 |\psi'(a)|=\frac1{1-|a|^2}
\]
and hence
\[
 |\det D\psi(a)|
 =
 \frac1{(1-|a|^2)^2}.
\]
This proves \eqref{local:eq:J2-exact}.  A direct calculation in
\eqref{local:eq:Lnrq-one-dimensional} gives
\[
 L_{2,r}(1)=\frac{2}{1-r^2}.
\]
Thus $q=1$ attains the minimum in \eqref{local:eq:Jn-exact-one-parameter}; by
uniqueness, $q_2(r)=1$.
\end{proof}

\begin{corollary}[Simply connected planar source domains]
\label{local:cor:planar-conformal-radius}
Let $\Omega\subsetneq\mathbb C$ be simply connected and let
$f:\Omega\to\mathbb D$ be harmonic.  If $r_\Omega(z)$ denotes the conformal
radius of $\Omega$ at $z$, then
\begin{equation}\label{local:eq:conformal-radius-bound-new}
 |\det Df(z)|
 \le
 \frac1{r_\Omega(z)^2}.
\end{equation}
The estimate is sharp.
\end{corollary}

\begin{proof}
Choose a conformal map $\phi:\mathbb D\to\Omega$ with $\phi(0)=z$.  By
definition, $r_\Omega(z)=|\phi'(0)|$.  The map $g=f\circ\phi$ is harmonic and
$|\det Dg(0)|\le1$.  Since
\[
 \det Dg(0)=\det Df(z)\det D\phi(0),
 \qquad
 |\det D\phi(0)|=r_\Omega(z)^2,
\]
we obtain \eqref{local:eq:conformal-radius-bound-new}.  Equality is attained by a
conformal map from $\Omega$ onto $\mathbb D$.
\end{proof}

\begin{remark}
Sharp Schwarz--Pick-type inequalities for planar harmonic diffeomorphisms,
with geometric constraints on the target, were also obtained by Kalaj
\cite{Kalaj2019Sharp}.
\end{remark}

The centered Euclidean estimate also follows directly from Theorem~\ref{thm:main}.  Letting $r\downarrow0$ in that theorem
gives
\[
 \int_{\Sph^{n-1}}|Df(0)^T\xi|\,d\sigma(\xi)\le1,
 \qquad
 |\det Df(0)|\le1,
\]
with equality for orthogonal maps.  The exact profile above shows what replaces
this centered determinant statement away from the origin.  In contrast, the
hyperbolic-harmonic invariance used in Theorem~\ref{thm:hyp-differential-intro} produces
a closed sharp all-point formula without the auxiliary minimization parameter.

\section{Two structural applications beyond harmonic mappings}
\label{sec:structural-applications}

We apply the zonal theorem to the Henyey--Greenstein scattering operator
and the trimmed-body theorem to zonoid depth regions.  Both applications
use the same sharp constants as the preceding results.

\subsection{Henyey--Greenstein scattering and geometric angular mixing}

The Henyey--Greenstein phase function was introduced in astrophysical
scattering \cite{HenyeyGreenstein} and remains a standard one-parameter
model for anisotropic radiative scattering; it is also widely used in
biomedical optics \cite{CalabroBigio}.  We restrict to the forward-scattering
regime $0\le g<1$, for which the angular kernel is decreasing in the
geodesic distance and therefore lies in the scope of
Theorem~\ref{thm:zonal-main}.  Its density with respect to the solid-angle
measure $d\Omega$ on $\mathbb S^2$ is
\[
 p_g(\xi,\eta)
 =
 \frac{1}{4\pi}
 \frac{1-g^2}
 {(1+g^2-2g\langle\xi,\eta\rangle)^{3/2}}.
\]
With our normalized spherical probability measure
$d\sigma=d\Omega/(4\pi)$, the associated Markov operator therefore takes
the form
\begin{equation}\label{eq:HG-operator}
 \mathsf H_gF(\xi)
 =
 \int_{\mathbb S^2}
 \frac{1-g^2}
 {(1+g^2-2g\langle\xi,\eta\rangle)^{3/2}}
 F(\eta)\,d\sigma(\eta).
\end{equation}
The kernel in \eqref{eq:HG-operator} is exactly the Poisson kernel of
$\mathbb B^3$ at radius $g$.

\begin{corollary}[Sharp geometric mixing for Henyey--Greenstein scattering]
\label{cor:HG}
Let $F:\mathbb S^2\to\overline{\mathbb B^3}$ be measurable and let
$0\le g<1$.  Put
\[
 K_g^{\mathrm{HG}}
 =
 \operatorname{co}\mathsf H_gF(\mathbb S^2).
\]
Then
\begin{equation}\label{eq:HG-width}
 w(K_g^{\mathrm{HG}})\le2g.
\end{equation}
Moreover, for $j=1,2,3$,
\begin{equation}\label{eq:HG-intrinsic}
 V_j(K_g^{\mathrm{HG}})
 \le
 g^j V_j(\overline{\mathbb B^3}).
\end{equation}
For $0<g<1$, equality in \eqref{eq:HG-width} holds if and only if
$F(\eta)=Q\eta$ almost everywhere for some $Q\in O(3)$.
\end{corollary}

\begin{proof}
If $g=0$, then the kernel in \eqref{eq:HG-operator} is identically $1$, so
\[
 \mathsf H_0F(\xi)=\int_{\mathbb S^2}F\,d\sigma
\]
is constant.  Hence $K_0^{\mathrm{HG}}$ is a singleton and
$w(K_0^{\mathrm{HG}})=V_j(K_0^{\mathrm{HG}})=0$ for $j=1,2,3$, which gives
\eqref{eq:HG-width}--\eqref{eq:HG-intrinsic} at the endpoint $g=0$.
Now assume $0<g<1$.  Equation \eqref{eq:HG-operator} is the
three-dimensional Poisson operator $P_g$.  Corollary~\ref{cor:PoissonWidth}
gives \eqref{eq:HG-width} and its rigidity statement.  The
intrinsic-volume bounds follow from \eqref{eq:extended-Urysohn}.
\end{proof}

The Poisson multipliers on spherical harmonics of degree $\ell$ are
$g^\ell$; see, for example, \cite{AxlerBourdonRamey,DaiXu}.  Consequently
the family $\{\mathsf H_g:0\le g<1\}$ satisfies
the multiplicative semigroup law
\begin{equation}\label{eq:HG-semigroup}
 \mathsf H_g\mathsf H_h=\mathsf H_{gh}.
\end{equation}
Consequently, for successive scattering steps with parameters
$g_1,\ldots,g_m$, the same bounds hold with $g$ replaced by
$G_m=\prod_{q=1}^m g_q$.

Thus the physical anisotropy parameter $g$, which is the first angular
moment of the Henyey--Greenstein phase function, is also the exact
convex-geometric contraction factor for bounded three-component angular
observables.  The statement is deterministic and does not require a
particular transport equation beyond the angular scattering operator
itself.

\subsection{Zonoid depth regions and multivariate expected shortfall}

The trimmed bodies in Section~3 are the geometric core of zonoid trimming.
Koshevoy and Mosler introduced zonoid trimmed regions as affine-equivariant
central regions for multivariate distributions \cite{KoshevoyMosler1997};
see also \cite{Mosler}.  In the risk-theoretic framework of Cascos and
Molchanov, zonoid trimming yields the multivariate counterpart of expected
shortfall \cite{CascosMolchanov}.  Our sharp spherical trimmed-body
inequality therefore gives dimension-dependent optimal envelopes for these
regions under a bounded-support assumption.

Let $X$ be an $\mathbb R^n$-valued random vector.  For
$0<\alpha\le1$, define its zonoid $\alpha$-trimmed region by
\begin{equation}\label{eq:zonoid-region}
 D_\alpha(X)
 =
 \left\{
 \mathbb E[\gamma X]:
 0\le\gamma\le\frac1\alpha,\quad
 \mathbb E\gamma=1
 \right\}.
\end{equation}
The definition depends only on the law of $X$: replacing $\gamma$ by
$\mathbb E[\gamma\mid X]$ preserves its bounds, its expectation, and
$\mathbb E[\gamma X]$.

\begin{corollary}[Sharp envelopes for bounded zonoid regions]
\label{cor:zonoid-envelope}
Let $n\ge2$ and suppose that
\[
 X\in b+R\overline{\mathbb B^n}
 \qquad\text{almost surely},
\]
where $b\in\mathbb R^n$ and $R>0$.  Then for every
$0<\alpha\le1$,
\begin{equation}\label{eq:zonoid-width}
 w(D_\alpha(X))
 \le
 2R\frac{\Phi_n(\alpha)}{\alpha}.
\end{equation}
More generally, for $j=1,\ldots,n$,
\begin{equation}\label{eq:zonoid-intrinsic}
 V_j(D_\alpha(X))
 \le
 R^j
 \left(\frac{\Phi_n(\alpha)}{\alpha}\right)^j
 V_j(\overline{\mathbb B^n}).
\end{equation}
The constants are sharp: equality is attained when
$X=b+RU$ and $U$ is uniformly distributed on $\mathbb S^{n-1}$.
\end{corollary}

\begin{proof}
Apply the probability-space form of the trimmed-body inequality from
Remark~\ref{rem:probability-trimmed}.  Put
\[
 Y=\frac{X-b}{R},
 \qquad |Y|\le1.
\]
If $a=\alpha\gamma$, then the constraints in
\eqref{eq:zonoid-region} become
\[
 0\le a\le1,
 \qquad
 \mathbb E a=\alpha,
\]
and therefore
\begin{equation}\label{eq:zonoid-trimmed-identity}
 D_\alpha(X)
 =
 b+\frac{R}{\alpha}\,\mathcal Z_\alpha^\Omega(Y).
\end{equation}
Mean width is translation invariant and homogeneous, so
Remark~\ref{rem:probability-trimmed} gives
\[
 w(D_\alpha(X))
 =
 \frac{R}{\alpha}w(\mathcal Z_\alpha^\Omega(Y))
 \le
 2R\frac{\Phi_n(\alpha)}{\alpha}.
\]
Applying \eqref{eq:extended-Urysohn} to
$\mathcal Z_\alpha^\Omega(Y)$ and using the homogeneity and translation
invariance of intrinsic volumes proves \eqref{eq:zonoid-intrinsic}.
If $Y$ is uniformly distributed on the sphere, then
$\Omega=\mathbb S^{n-1}$ and the sharpness part of
Theorem~\ref{thm:trimmed} gives
\[
 \mathcal Z_\alpha^\Omega(Y)
 =
 \Phi_n(\alpha)\overline{\mathbb B^n},
\]
which yields equality throughout.
\end{proof}

In dimension three, $\Phi_3(\alpha)=\alpha(1-\alpha)$, so the result takes
the particularly simple form
\begin{equation}\label{eq:zonoid-3d}
 w(D_\alpha(X))\le2R(1-\alpha),
 \qquad
 V_j(D_\alpha(X))
 \le
 R^j(1-\alpha)^jV_j(\overline{\mathbb B^3}).
\end{equation}
Thus bounded multivariate distributions have a sharp, distribution-free
convex-geometric envelope for every zonoid depth level.  Through the
zonoid--expected-shortfall correspondence of
\cite{CascosMolchanov}, the same inequalities control the geometry of the
associated set-valued risk regions.

\section*{Acknowledgements}
The work of D.~Zhong was supported by the Guangdong Basic and Applied Basic Research Foundation (Nos.~2022A1515110967 and 2023A1515011809).  D.~Kalaj gratefully acknowledges financial support from the Ministry of Education, Science and Innovation of Montenegro through the grants ``Mathematical Analysis, Optimisation and Machine Learning'' and ``Complex-analytic and geometric techniques for non-Euclidean machine learning: theory and applications.''

\section*{AI Declaration}
AI-assisted tools were used during the preparation of the manuscript for editorial, organizational, and typesetting assistance.  The authors take full responsibility for the mathematical statements, derivations, citations, and conclusions in the final manuscript.

\end{document}